\documentclass[10pt,oneside]{amsart}

\usepackage{amssymb, amsmath, amsthm, amsfonts}
\usepackage{mathrsfs,comment}
\input{xypic}
\usepackage{graphicx}
\usepackage{geometry}
\usepackage{float}

\usepackage[bookmarks,bookmarksopen,bookmarksdepth=4]{hyperref}
\usepackage{url}
\usepackage{enumerate}
\usepackage{color}

\usepackage{verbatim}
\usepackage[T1]{fontenc}

\hypersetup{%
  bookmarksnumbered=true,%
  colorlinks=true,%
  linkcolor=blue,%
  citecolor=blue,%
  filecolor=blue,%
  menucolor=blue,%
  urlcolor=blue,%
  pdfnewwindow=true,%
  pdfstartview=FitBH}

\def\@url#1{{\tt\def~{\lower3.5pt\hbox{\char'176}}\def\_{\char'137}#1}}

\newtheorem{thm}{Theorem}[section]
\newtheorem{cor}[thm]{Corollary}
\newtheorem{prop}[thm]{Proposition}
\newtheorem{lem}[thm]{Lemma}
\theoremstyle{definition}
\newtheorem{defn}[thm]{Definition}
\newtheorem{exmp}[thm]{Example}

\newtheorem{example}[thm]{Example}

\newtheorem{rem}[thm]{Remark}

\newtheorem{construction}[thm]{Construction}

\makeatletter
\let\c@lem=\c@thm
\let\c@cor=\c@thm
\let\c@prop=\c@thm
\let\c@lem=\c@thm
\let\c@defn=\c@thm
\let\c@exmps=\c@thm
\let\c@rem=\c@thm
\let\c@warn=\c@thm
\let\c@claim=\c@thm
\makeatother

\numberwithin{equation}{subsection}

\def\endash{\mathchar"707B}
\newcommand{\RR}{\mathbb{R}}
\newcommand{\cell}{\endash \textnormal{cell} \endash}

\newcommand{\leftmod}{\endash \textnormal{mod}}
\newcommand{\leftdmod}{\endash \textnormal{dmod}}

\newcommand{\modin}{\endash \textnormal{mod} \endash}
\newcommand{\acal}{ {\mathcal{A}} }
\newcommand{\cKS}{ {\mathcal{K}} }

\newcommand{\modcat}[1]{#1 \endash \textnormal{mod}}

\newcommand{\RRti}{\widetilde{\RR}}

\newcommand{\Mdef}[2]{\newcommand{#1}{\relax \ifmmode #2 \else $#2$\fi}}
\Mdef{\bA}{\mathbb{A}}
\Mdef{\bB}{\mathbb{B}}
\Mdef{\bC}{\mathbb{C}}
\Mdef{\bD}{\mathbb{D}}
\Mdef{\bE}{\mathbb{E}}
\Mdef{\bF}{\mathbb{F}}
\Mdef{\bG}{\mathbb{G}}
\Mdef{\bH}{\mathbb{H}}
\Mdef{\bI}{\mathbb{I}}
\Mdef{\bJ}{\mathbb{J}}
\Mdef{\bK}{\mathbb{K}}
\Mdef{\bL}{\mathbb{L}}
\Mdef{\bM}{\mathbb{M}}

\newcommand{\bN}{\mathbb{N}}

\Mdef{\bO}{\mathbb{O}}
\Mdef{\bP}{\mathbb{P}}
\Mdef{\bQ}{\mathbb{Q}}
\Mdef{\bR}{\mathbb{R}}
\Mdef{\bS}{\mathbb{S}}
\Mdef{\bT}{\mathbb{T}}
\Mdef{\bU}{\mathbb{U}}
\Mdef{\bV}{\mathbb{V}}
\Mdef{\bW}{\mathbb{W}}
\Mdef{\bX}{\mathbb{X}}
\Mdef{\bY}{\mathbb{Y}}
\Mdef{\bZ}{\mathbb{Z}}

\Mdef{\bbS}{\mathbb{S}}

\Mdef{\mcA}{\mathcal{A}}
\Mdef{\mcB}{\mathcal{B}}
\Mdef{\mcC}{\mathcal{C}}
\Mdef{\mcD}{\mathcal{D}} 
\Mdef{\mcE}{\mathcal{E}}
\Mdef{\mcF}{\mathcal{F}}
\Mdef{\mcG}{\mathcal{G}}
\Mdef{\mcH}{\mathcal{H}} 
\Mdef{\mcI}{\mathcal{I}}
\Mdef{\mcJ}{\mathcal{J}}
\Mdef{\mcK}{\mathcal{K}}
\Mdef{\mcL}{\mathcal{L}}

\Mdef{\mcM}{\mathcal{M}}
\Mdef{\mcN}{\mathcal{N}}
\Mdef{\mcO}{\mathcal{O}}
\Mdef{\mcP}{\mathcal{P}}
\Mdef{\mcQ}{\mathcal{Q}}
\Mdef{\mcR}{\mathcal{R}}
\Mdef{\mcS}{\mathcal{S}}
\Mdef{\mcT}{\mathcal{T}}
\Mdef{\mcU}{\mathcal{U}}
\Mdef{\mcV}{\mathcal{V}}
\Mdef{\mcW}{\mathcal{W}}
\Mdef{\mcX}{\mathcal{X}}
\Mdef{\mcY}{\mathcal{Y}}
\Mdef{\mcZ}{\mathcal{Z}}

\Mdef{\At}{\tilde{A}}
\Mdef{\Bt}{\tilde{B}}
\Mdef{\Ct}{\tilde{C}}
\Mdef{\Et}{\tilde{E}}
\Mdef{\Ht}{\tilde{H}}
\Mdef{\Kt}{\tilde{K}}
\Mdef{\Lt}{\tilde{L}}
\Mdef{\Mt}{\tilde{M}}
\Mdef{\Nt}{\tilde{N}}
\Mdef{\Pt}{\tilde{P}}

\Mdef{\tA}{\tilde{A}}
\Mdef{\tB}{\tilde{B}}
\Mdef{\tC}{\tilde{C}}
\Mdef{\tE}{\tilde{E}}
\Mdef{\tH}{\tilde{H}}
\Mdef{\tK}{\tilde{K}}
\Mdef{\tL}{\tilde{L}}
\Mdef{\tM}{\tilde{M}}
\Mdef{\tN}{\tilde{N}}
\Mdef{\tP}{\tilde{P}}

\Mdef{\ft}{\tilde{f}}
\Mdef{\xt}{\tilde{x}}
\Mdef{\yt}{\tilde{y}}

\Mdef{\Ab}{\overline{A}}
\Mdef{\Bb}{\overline{B}}
\Mdef{\Cb}{\overline{C}}
\Mdef{\Db}{\overline{D}}
\Mdef{\Eb}{\overline{E}}
\Mdef{\Fb}{\overline{F}}
\Mdef{\Gb}{\overline{G}}
\Mdef{\Hb}{\overline{H}}
\Mdef{\Ib}{\overline{I}}
\Mdef{\Jb}{\overline{J}}
\Mdef{\Kb}{\overline{K}}
\Mdef{\Lb}{\overline{L}}
\Mdef{\Mb}{\overline{M}}
\Mdef{\Nb}{\overline{N}}
\Mdef{\Ob}{\overline{O}}
\Mdef{\Pb}{\overline{P}}
\Mdef{\Qb}{\overline{Q}}
\Mdef{\Rb}{\overline{R}}
\Mdef{\Sb}{\overline{S}}
\Mdef{\Tb}{\overline{T}}
\Mdef{\Ub}{\overline{U}}
\Mdef{\Vb}{\overline{V}}
\Mdef{\Wb}{\overline{W}}
\Mdef{\Xb}{\overline{X}}
\Mdef{\Yb}{\overline{Y}}
\Mdef{\Zb}{\overline{Z}}

\Mdef{\db}{\overline{d}}
\Mdef{\hb}{\overline{h}}
\Mdef{\qb}{\overline{q}}
\Mdef{\rb}{\overline{r}}
\Mdef{\tb}{\overline{t}}
\Mdef{\ub}{\overline{u}}
\Mdef{\vb}{\overline{v}}

\Mdef{\hc}{\hat{c}}
\Mdef{\he}{\hat{e}}
\Mdef{\hf}{\hat{f}}
\Mdef{\hA}{\hat{A}}
\Mdef{\hH}{\hat{H}}
\Mdef{\hJ}{\hat{J}}
\Mdef{\hM}{\hat{M}}
\Mdef{\hP}{\hat{P}}
\Mdef{\hQ}{\hat{Q}}

\Mdef{\thetab}{\overline{\theta}}
\Mdef{\phib}{\overline{\phi}}

\Mdef{\uA}{\underline{A}}
\Mdef{\uB}{\underline{B}}
\Mdef{\uC}{\underline{C}}
\Mdef{\uD}{\underline{D}}

\Mdef{\bolda}{\mathbf{a}}
\Mdef{\boldb}{\mathbf{b}}
\Mdef{\boldD}{\mathbf{D}}

\newcommand{\elr}[1]{E\langle #1\rangle}

\newcommand{\lra}{\longrightarrow}
\newcommand{\spO}{\mathrm{Sp}^{O}}
\newcommand{\GspO}{G\mathrm{Sp}^{O}}
\newcommand{\NspO}{\bN \mathrm{Sp}^{O}}
\newcommand{\WspO}{W\mathrm{Sp}^{O}}
\newcommand{\TspO}{\torus \mathrm{Sp}^{O}}

\newcommand{\Tsub}{\mathfrak{T}}

\Mdef{\infl}{\mathrm{inf}}
\Mdef{\defl}{\mathrm{def}}
\Mdef{\res}{\mathrm{res}}
\Mdef{\ind}{\mathrm{ind}}
\Mdef{\coind}{\mathrm{coind}}
\Mdef{\id}{\mathrm{Id}}

\newcommand{\Rinf}{\widetilde{\infl}_W^N}

\newcommand{\adjunction}[4]{
\xymatrix{
#1:#2 \ar@<0.7ex>[r] &
\ar@<0.7ex>[l] #3:#4
}}

\newcommand{\piAG}{\pi^{\mcA (G)}}
\newcommand{\piAN}{\pi^{\mcA (\bN)}}
\newcommand{\piAT}{\pi^{\mcA (\torus)}}
\newcommand{\piA}{\pi^{\mcA}}

\definecolor{orange}{RGB}{255, 127, 0}

\definecolor{darkgreen}{RGB}{35, 127, 89}

\newcommand{\T}{ {\mathbb{T}} }
\newcommand{\torus}{ {\mathbb{T}} }

\newcommand{\codim}{\mathrm{codim}}

\newcommand{\sm }{\wedge}

\newcommand{\tensor}{\otimes}

\newcommand{\Hom}{\mathrm{Hom}}

\newcommand{\Ext}{\mathrm{Ext}}
\Mdef{\bhom}{\mathbf{\hat{H}om}}

\Mdef{\Mod}{\mathrm{mod}}

\newcommand{\st}{\; | \;}

\newcommand{\PCf}{PC_f}

\newcommand{\cEi}{\mcE^{-1}}

\newcommand{\RRhat}{\hat{\bR}}

\newcommand{\Gspectra}{\mbox{$G$--spectra}}

\newcommand{\Wt}{\tilde{W}}
\newcommand{\Sc}{\Sigma^c}
\newcommand{\bfb}{\mathbf{b}}
\newcommand{\bfd}{\mathbf{d}}
\newcommand{\bfe}{\mathbf{e}}
\newcommand{\bff}{\mathbf{f}}

\newcommand{\Q}{\mathbb {Q}}
\newcommand{\R}{\mathbb {R}}
\newcommand{\toralGspectra}{\mbox{toral--$G$--spectra}_{\Q}}
\newcommand{\toralNspectra}{\mbox{toral--$\bN$--spectra}_{\Q}}

\newcommand{\efp}{E\mcF_+}

\newcommand{\efhp}{E\mcF/H_+}
\newcommand{\efkp}{E\mcF/K_+}
\newcommand{\eflp}{E\mcF/L_+}

\newcommand{\siftyV}[1]{S^{\infty V(#1)}}

\newcommand{\sifty}[1]{S^{\infty V(#1)}}

\newcommand{\qqed}{\qed \\[1ex]}

\newcommand{\holim}{\mathop{ \mathop{\mathrm {holim}} \limits_\leftarrow} \nolimits}

\newcommand{\RRt}{\RR_t}
\newcommand{\RRa}{\RR_a}

\newcommand{\RRtop}{\RR_{top}}

\newcommand{\RRttop}{\widetilde{\RR}_{top}}

\newcommand{\RRtmod}{\modcat{\RRt}}

\newcommand{\RRtopmod}{\modcat{\RRtop}}

\renewcommand{\Et}{\mcE_t}

\newcommand{\cOcF}{\mcO_{\mcF}}

\newcommand{\cOcFK}{\mcO_{\mcF /K}}
\newcommand{\cOcFL}{\mcO_{\mcF /L}}

\newcommand{\fibrep}{\widehat{f}}

\newcommand{\quillen}{\underset{Q}{\simeq}}

\newcommand{\fX}{\mathfrak{X}}

\newcommand{\C}{\mathbb{C}}

\title{An algebraic model for rational toral $G$--spectra}
\author[Barnes]{David Barnes}
\address[Barnes]{Mathematical Sciences Research Centre, Queen's University Belfast}
\author[Greenlees]{J.P.C. Greenlees}
\address[Greenlees]{Warwick Mathematics Institute, Zeeman
  Building, Coventry, CV4 7AL, UK}

\author[K\k{e}edziorek]{Magdalena K\k{e}dziorek}
\address[K\k{e}dziorek]{Max Planck Institute for Mathematics, Bonn}

\begin{document}
\begin{abstract}
For $G$ a compact Lie group,  toral $G$--spectra are those rational
$G$--spectra whose geometric isotropy consists of  subgroups of a
maximal torus of $G$. The homotopy category of rational
toral $G$--spectra is a retract  of the category of all rational $G$--spectra.

In this paper we show that the abelian category of \cite{AGtoral}
gives an  algebraic model for the toral part of rational $G$--spectra.
This is a major step in establishing an algebraic model
for all rational $G$--spectra, for any compact Lie group $G$.
\end{abstract}

\maketitle

\tableofcontents

\newpage
\section{Introduction}

\subsection{Overview}
This paper is part of an overarching project to give algebraic
models for the category of rational $G$--spectra for any
compact Lie group $G$ and hence describe the triangulated
category of rational $G$--equivariant cohomology
theories. Specifically, in this paper we give an algebraic model for
the toral $G$--spectra (those concentrated over subgroups of a maximal torus).

In a little more detail,  the conjecture due to the second author \cite{Gmfo} states that for every compact Lie group $G$ there is a graded abelian
category $\mcA (G)$ and a Quillen equivalence
$$\Gspectra \quillen d\mcA(G)$$
between rational $G$--spectra and differential  objects in $\mcA(G)$.
The abelian model $\mcA(G)$ is conjectured to be a category of sheaves over
the space of closed subgroups of $G$ with the stalk over a subgroup
$K$ giving the information over the isotropy group $K$.
For example the stalk over the trivial subgroup gives a model for free
$G$-spectra.

In this paper we focus on rational $G$--spectra with isotropy consisting of subgroups of a maximal torus $\torus$. There are
several reasons for selecting this part: crudely, it is both useful
and accessible. It is useful since it includes many
important examples ($K$-theory, elliptic cohomology, Borel
cohomology,.....) and it is a direct summand  of the whole
category capturing much of the character of the whole. On the other hand it is fairly accessible because it is
relatively easy to describe the subgroups of a torus and their
conjugacy in $G$.

When $G=\torus$ is a torus, all $\torus$--spectra are toral.
In this case, the category $\mcA (\torus)$ has been made explicit \cite{tnq1},
and the conjecture has been proved in \cite{tnqcore}. Moving on, one may then
want to mimic complex representation theory of a compact
Lie group, when one understands representations by restricting them to the
maximal torus,  whilst remembering the action of the Weyl group. Of
course, when restricting to $\torus$  it is only
reasonable to hope to capture information about subconjugates of $\torus$,
so that the toral information is the most we can hope for. Moreover,
 it was proved in \cite{AGtoral} that {\em all} information about all toral subgroups
is captured by restriction to $\torus$. An abelian model
$\mcA (G,toral)$ for toral $G$--spectra was constructed from the model
$\mcA (\torus)$ for the maximal torus $\torus$ by fully taking into
account the action of the Weyl group $W=N_G(\torus)/\torus$ on
$\torus$ and all its subgroups.  The paper \cite{AGtoral} constructed
a finite Adams spectral sequence based on $\mcA (G, toral)$ convergent
for all toral spectra.  Our purpose
here is to show that $\mcA (G, toral)$ can be used to give a full
algebraic model of the category of toral $G$--spectra. In other words, we prove here the conjecture from \cite{AGtoral} that the category of rational toral $G$--spectra is Quillen equivalent to the category of differential objects in $\mcA(G, toral)$.

The toral information is captured by the maximal torus for
any group $G$, and in particular this applies to the normalizer of the
maximal torus $\bN=N_G(\torus)$. This is the simplest of the groups with
maximal torus $\torus$ and Weyl group $W$. It was shown in \cite{AGtoral}
that the category of rational toral spectra for $G$ is a retract of
the category of rational toral spectra for $\bN$.   In fact, we may deduce that toral $G$--spectra
and $\mcA (G, toral)$ are obtained from toral $\bN$--spectra and
$\mcA(\bN,toral)$ by cellularization at a collection of toral cells that we denote
$G/\Tsub_+$ and define later. This observation suggests the idea of the proof for general $G$; first we show that
\[
\toralNspectra \quillen d\mcA (\bN, toral)
\]
and later deduce the sequence of Quillen equivalences
\[
\toralGspectra
\quillen
G/\Tsub_+ \cell \toralNspectra
\quillen
G/\Tsub_+ \cell d\mcA (\bN, toral)
\quillen
d\mcA (G, toral)
\]
by applying cellularization to the first statement, see Diagram \ref{fig:main} for more details.

Finally, since $\torus$ is the identity component of $\bN$ it is not hard to
see that in both topology (toral $\bN$--spectra) and in algebra  (the
abelian category $\mcA(\bN, toral)$), the categories consist of their $\torus$--equivariant
counterpart together with  an action of the Weyl group $W$ (in a
pervasive way that includes  the action of $W$ on the subgroups).
This observation suggests the strategy of the proof for $\bN$; we take the proof of
\cite{tnqcore} for the torus, and show that it  is compatible with
the action of the Weyl group in the sense that the arguments can be
lifted from the category of $\torus$--spectra to toral $\bN$--spectra.

The general idea of the proof is straightforward, but considerable
work is involved in implementing it.  We proceed with describing the issues below.

\subsection{The challenges}

To describe the novel contribution of this paper, we need to explain
in more detail the methods employed in earlier work.

The reduction from $G$ to the normalizer $\bN$ of its maximal torus
(i.e. from the general case to the case with identity component a
torus) applies the cellularization principle of \cite{gscell} to the
restriction-coinduction adjunction. The reason that this works is
that all toral subgroups are represented inside $\bN$  and furthermore
all conjugates of toral subgroups are represented inside $\bN$.
This is in addition to the fundamental fact that
the spectrum $\Sigma^\infty G/\bN_+$ is rationally trivial as a toral
spectrum. This
part of the argument was first described in \cite{KedziorekSO(3)} for the case of $G=SO(3)$. The identification of
$d\mcA (G,toral)$ with the cellularization of $d\mcA (\bN,toral)$ is then
deduced from the results of \cite{AGtoral}.

The step described above reduces us to the case of a group $\bN$ with identity component
$\torus$. To describe this case, we now recall the strategy in the case of a torus from \cite{tnqcore}. We start by identifying
$\torus$--spectra with modules over the sphere spectrum $\bbS$. The
first substantive step is to express $\bbS$ as a homotopy pullback of
a punctured $(r+1)$--cube diagram $\RRti$ of ring $\torus$--spectra, where $r$
is the rank of $\torus$. It follows that the
category of modules over $\bbS$ can be modeled as the cellularization
of the category of modules over the diagram $\RRti$. The advantage is that
each of the ring $\torus$--spectra $\RRti (a)$ in the diagram is determined by its homotopy, and its module category can be understood in algebraic
 terms and ultimately related to the category $\mcA (\torus)$.

We implement here the same argument for $\bN$. The $\bN$--equivariant
sphere should be torally equivalent to the homotopy pullback of an
$(r+1)$--cube of ring $\bN$--spectra. In fact, because of the way toral
equivalence is defined, this will be immediate if we can
construct the terms $\RRti (a)$ (for $a$ a vertex in the cube)
as $\bN$--spectra. This is one of the
major technical achievements of this paper.

The terms $\RRti (a)$ can be viewed as constructed
(using products and localizations) from  stalks $\RRti_K$
corresponding to the  subgroups $K$ of $\torus$. Our main challenge is that
the group $K$ is only fixed by the normalizer $N_\bN(K)$, a group that contains
$\torus$ but will typically not be all of $\bN$. Our first step is to
observe that   $\RRti_K$ can be constructed as a ring $N_\bN(K)$--spectrum (and
not just as a $\torus$--spectrum). The substantial new step in the
argument is the proof that the object constructed from all
of these stalks (not individually $\bN$--spectra) does in fact admit an action of all of $\bN$.
This will be done in Sections \ref{sec:equivcoind} and \ref{sec:theNequivcube}.

\subsection{Precise constructions}
Finally we will be more precise about the categories of objects we are
discussing.

The endomorphism ring of the sphere spectrum is the rational Burnside
ring $[\bbS , \bbS]^G=A(G)$, and tom Dieck showed  that this can
be understood using the mark homomorphism. Indeed, if $f:\bbS\lra
\bbS$ then we may take geometric $H$--fixed points to get a
non-equivariant map $\Phi^Hf: \bbS \lra \bbS$, and we write $m(f)(H)$
for its degree. This defines a function $m(f): \Phi G \lra \Q$ from
the space $\Phi G$ of conjugacy classes of subgroups $H$ of $G$ with finite
index in their normalizers. If we give $\Phi G$  the quotient of the Hausdorff metric
topology, $m(f)$ is continuous, and in fact
\[
m: [\bbS , \bbS]^G\stackrel{\cong}\lra C(\Phi G, \Q)
\]
is an isomorphism \cite[5.6.4, 5.9.13]{tomdieck1}. Here $\Q$ on the right has a discrete topology and $C$ denotes continuous functions. Accordingly any open and closed subset of $\Phi G$
defines an idempotent of $A(G)$ with this support. In particular, the
conjugacy class of the maximal torus is open and closed in $\Phi G$. We write
$e_\torus$ for the associated idempotent in $A(G)$. Evidently
$e_\torus\bbS \simeq E\Lambda (\torus)_+$ where $\Lambda (\torus)$ is the family of all
subgroups of a maximal torus and $E\Lambda (\torus)_+$ denotes the universal $G$--space for the family $\Lambda(\torus)$. The homotopy category of toral
spectra consists of those of the form $e_\torus X\simeq E\Lambda (\torus)_+\sm
X$.

We will be working with model categories, so we should specify our
chosen models.  To start with, our model for $G$--spectra is the category $\GspO$ of
orthogonal $G$--spectra. We rationalize it by localizing at the rational
sphere spectrum $\bbS_\Q$. We form the category of toral $G$--spectra by
localizing at the idempotent $e_\torus$. These can be done together by
localizing with respect to ${e_\torus \bbS_\Q}$.

At the algebraic end, we will not give a detailed description of the
models. Instead we state that the abelian category $\mcA (\torus)$ is
defined for a torus $\torus$ in \cite{tnq1}. The abelian category
$\mcA (G, toral)$ is defined in \cite{AGtoral}. These categories are of
finite injective dimension, and the categories $d\mcA (\torus)$ and $d\mcA
(G, toral)$ of differential graded objects admit injective model structures, as
described in \cite{tnqcore}. Perhaps the most efficient way
to show these injective model structures exist is to use the left-lifting technique of \cite{HKRS}.

With these categories specified, the other notation is easily inferred
using some general constructions. First of all, we adopt the general
conventions of \cite{tnqcore}, since we often consider
corresponding rings in different categories. Thus $\tilde{R}_{top}$ might
be a ring $G$--spectrum, $R_{top}$ the corresponding ring in spectra,
$R_t$ the corresponding ring in DGAs, and $R_a$ in graded rings; this
convention is only suggestive, and the exact definitions will need to
be given in due course.

Coming to model structures,  if $\tilde{R}_{top}$ is a ring in a category of $G$--spectra
the category $\tilde{R}_{top} \leftmod$ denotes the category
of $\tilde{R}_{top}$--modules with the
algebraically projective model structure generated by cells
$G/H_+\sm \tilde{R}_{top}$ (more properly the model structure right lifted from
$G$--spectra along the forgetful functor).
The algebraic side is similar, but here we also use the injective model structure (left-lifted along the forgetful functor). It will be clearly indicated which model structure we have in mind.

We also need to consider diagrams $\RRttop$ of rings in
$G$--spectra, we write $\RRttop \leftmod$ for the category of diagrams
of modules with the diagram-injective model structure (in which weak
equivalences and cofibrations are created by the evaluations at the
vertices).

From these categories we form other model structures by localization.
We write  $L_E$ to denote the left Bousfield-localization with
respect to the object $E$ and $A$-cell to denote cellularization (or right
Bousfield localization) with respect to a set of objects $A$.

\subsection{The strategy}
The plan for this paper is to prove the Quillen equivalences of
Figure \ref{fig:main}.
In that diagram left adjoints are on the left
and double ended arrows indicate zig-zags of Quillen equivalences.

The first step in this paper is to prove that
the model category of toral $G$--spectra is Quillen equivalent to
toral $\bN$--spectra, cellularized at the
set $G/\Tsub_+$ of $\bN$--spectra
$G/K_+$ for $K\subseteq \torus$.

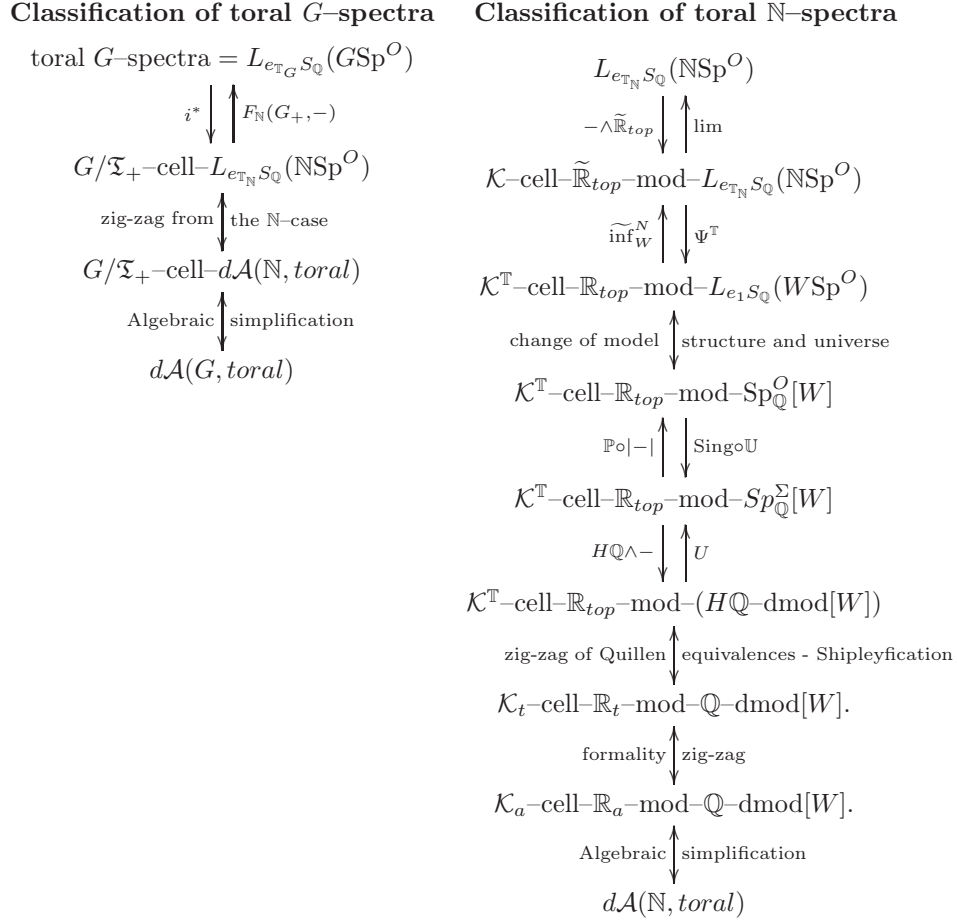
\begin{figure}[!ht]
\caption{Diagram of Quillen equivalences}\label{fig:main}
\begin{minipage}[t]{0.4\textwidth}
\medskip
\begin{center}
\textbf{Classification of toral $G$--spectra}
\end{center}
\[
\xymatrix@R=1.8pc{
\textrm{toral $G$--spectra} = L_{e_{\torus_G}S_{\bQ}}(\GspO)
\ar@<-1ex>[d]_{i^{\ast}}
\\
G/\Tsub_+ \cell L_{e_{\torus_\bN}S_{\bQ}}(\NspO)
\ar@<-1ex>[u]_{F_\bN(G_+,-)}
\ar@<-0ex>[d]_{\textnormal{zig-zag from}}
\\
G/\Tsub_+ \cell d\mcA(\bN,toral)
\ar@<-0ex>[u]_{\textnormal{the $\bN$--case}}
\ar@<-0ex>[d]_{\textnormal{Algebraic}}
\\
d\mcA(G,toral)
\ar@<-0ex>[u]_{\textnormal{simplification}}
}
\]
\end{minipage}
\begin{minipage}[t]{0.4\textwidth}
\medskip
\begin{center}
\textbf{Classification of toral $\bN$--spectra}
\end{center}
\medskip
\[
\xymatrix@R=1.8pc{
L_{e_{\torus_\bN}S_{\bQ}}(\NspO)
\ar@<-1ex>[d]_{ - \wedge \RRttop}
\\
\cKS \cell \RRttop \modin L_{e_{\torus_\bN}S_{\bQ}}(\NspO)
\ar@<-1ex>[u]_{ \lim}
\ar@<+1ex>[d]^{\Psi^\torus}
\\
\cKS^\torus \cell \RRtop \modin L_{e_1 S_\bQ}(\WspO)
\ar@<+1ex>[u]^{ \Rinf}
\ar@<+0ex>[d]_{\textnormal{change of model }}
\\
\cKS^\torus \cell \RRtop \modin  \spO_\bQ[W]
\ar@<-0ex>[u]_{\textnormal{structure and universe}}
\ar@<+1ex>[d]^{ \mathrm{Sing}\circ \mathbb{U}}
\\
\cKS^\torus \cell \RRtop \modin Sp_\bQ^\Sigma[W]
\ar@<+1ex>[u]^{\mathbb{P} \circ |-|}
\ar@<-1ex>[d]_{H\bQ\wedge -}
\\
\cKS^\torus \cell \RRtop \modin (H\bQ \leftdmod[W])
\ar@<-1ex>[u]_{U}
\ar@<-0ex>[d]_{\textnormal{zig-zag of Quillen}}
\\
\cKS_t \cell \RRt \modin \bQ \leftdmod[W].
\ar@<-0ex>[u]_{\textnormal{equivalences - Shipleyfication}}
\ar@<-0ex>[d]_{\textnormal{formality}}
\\
\cKS_a \cell \RRa \modin \bQ \leftdmod[W].
\ar@<-0ex>[u]_{\textnormal{zig-zag}}
\ar@<-0ex>[d]_{\textnormal{Algebraic}}
\\
d\mcA(\bN,toral)
\ar@<-0ex>[u]_{\textnormal{simplification}}
}
\]
\end{minipage}
\end{figure}

The next step is to classify toral $\bN$--spectra in terms of
an algebraic model. We show that there is a series of
Quillen equivalences between toral $\bN$--spectra,
$L_{e_{\torus_\bN}S_{\bQ}}(\NspO)$ and
$d\mcA(\bN,toral)$.
This part uses a number of cellularisations.
In each case the cellularisation is done at the set $\cKS=\bN/\Tsub_+$
 of the (derived) images of the generators
for toral $\bN$--spectra, under the
(derived) composite of functors from
toral $\bN$--spectra to the given category.
This step models rational toral $\bN$--spectra.

The final step for the classification of toral $G$--spectra
is to cellularize the whole diagram for toral $\bN$--spectra
at the (derived) images of cells $G/\Tsub_+$ of $L_{e_{T_G}S_{\bQ}}(\GspO)$ and at the end recognise the model category as $d\mcA(G,toral)$ with the injective model structure.

\subsubsection*{Acknowledgements} The authors wish to thank the Mathematisches Forschungsinstitut Oberwolfach for
providing ideal environment to work on this project as part of the Research in
Pairs Programme.

\section{Reduction to \texorpdfstring{$\bN$}{N}--spectra}

As mentioned in the introduction,
the rational Burnside ring $A(G)$ of an arbitrary compact Lie group $G$
has an idempotent $e_{\torus}$ corresponding to the maximal torus $\torus$ and all
its subconjugates, see \cite{greratmack}.
We may Bousfield localise the category of orthogonal $G$--spectra at
$e_{\torus}S_\bQ$, where $S_\bQ$ is the rational sphere spectrum.

\begin{defn}
The model category of \textbf{toral--$G$--spectra} is
the model category
\[
L_{e_{\torus}S_\bQ} \GspO.
\]
\end{defn}
This model category is Quillen equivalent to the
model structure on $G$--spectra right lifted
from $L_{S_\bQ} \TspO$ using the restriction functor
\[
\res^G_{T}: \GspO \lra \TspO.
\]
It is also Quillen equivalent to the cellularization of
rational $G$--spectra at the set of (all suspensions and desuspensions of) objects
$G/K_+$ as $K$ runs over the subconjugates of $\torus$.
In particular, the weak equivalences in toral--$G$--spectra
are precisely those maps which forget to weak equivalences of
rational $\torus$--spectra.

Thus a toral $G$--spectrum
is determined by its restriction to the maximal torus and
fusion information. All the relevant fusion information
is contained in $\bN=N_G(\torus)$. In effect, everything is organized
around the restriction functor
\[
\res^G_{\bN}: \GspO \lra \NspO.
\]
The restriction functor has both a left and a right adjoint, and it is
the right adjoint that is most important to us. Writing $i: \bN \lra
G$ for the inclusion, so that $i^*=\res^G_{\bN}$ the core structure is
$$\adjunction{i^*}{\GspO}{\NspO}{F_{\bN}(G_+, -)}$$
where $F_{\bN}(G_+, -)$ is the coinduction functor.
We will prove that with suitable model structures, this
adjunction gives a Quillen equivalence.
That this adjunction could be made into a Quillen equivalence
was central to the classification of rational $SO(3)$--spectra,
see \cite{KedziorekSO(3)}.

We begin with the model structures of toral $G$--spectra and
toral $\bN$--spectra, using idempotents
$e_{\torus_G} \in A(G)$ and
$e_{\torus_\bN} \in A(\bN)$ respectively.
The above adjunction gives a Quillen adjunction
\[
\adjunction{i^*}{L_{e_{\torus_G}S_{\bQ}}(\GspO)}{L_{e_{\torus_\bN}S_{\bQ}}(\NspO)}{F_{\bN}(G_+, -)}
\]
as the left adjoint preserves cofibrations and weak equivalences (this follows from the fact that $e_{\torus_\bN}i^*(e_{\torus_G})=e_{\torus_\bN}$ and $i^*$ is strong monoidal).

Now we apply the Cellularization Principle of \cite{gscell} to get a model for
toral $G$--spectra in terms of $\bN$--spectra.

\begin{thm}
\label{thm:toralGspectraisGcelltoralNspectra}
The restriction--coinduction adjunction induces a Quillen equivalence
\[
L_{e_{\torus_G}S_{\bQ}}(\GspO) \quillen G/\Tsub_+ \cell L_{e_{\torus_\bN}S_{\bQ}}(\NspO)
\]
where $G/\Tsub_+$ is the set of $\bN$--spectra $G/K_+$ for $K$ a subgroup of $\torus$.
\end{thm}

\begin{proof}
The essential information is that the spaces $G/K_+$ for
$K\subseteq \torus$ give a set of homotopically compact generators of the category of toral
$G$--spectra. That they are generators is clear from the description of the weak equivalences
and that they are homotopically compact follows as localization at an idempotent is smashing.
Similarly, $i^*(G/K_+)$ for $K\subseteq \torus$ is homotopically compact in $L_{e_{\torus_\bN}S_{\bQ}}(\NspO)$
since a compact $\bN$--manifold admits the structure of a finite $\bN$-CW-complex by
Illmann \cite{illman}.

The result now follows directly from the Cellularization Principle of \cite{gscell}
once we check that the components of the derived unit on generators
$G/K_+\lra F_{\bN}(G_+, G/K_+)$
are  weak equivalences in $L_{e_{\T_G}S_{\bQ}}(G-\spO)$.
This follows from \cite[Corollary 7.11]{AGtoral} which states that the unit of the algebraic counterpart of the derived adjunction of restriction and coinduction on rational toral $G$ and $\bN$ spectra is an isomorphism.
\end{proof}

\section{The sphere as a formal  pullback (recollections from the case
of a torus)}
\label{sec:liftedformalcube}

Central to the case of the torus $\torus$ in \cite{tnqcore}, was that the
$\torus$--equivariant sphere $\bS$ can be constructed as the pullback
of a cubical diagram of commutative ring $\torus$--spectra (called there the
`formal cube'). At each point of the cube other than where $\bS$ sits,  the ring
$\torus$--spectrum is obtained from isotropically simple ring $\torus$--spectra by a
process of localizations and products.

The purpose of Sections \ref{sec:equivcoind} and \ref{sec:theNequivcube}
is to show that the diagram
$\RRti$ of $\torus$--spectra can be lifted to a diagram $\RRttop$ of
$\bN$--spectra.
It then follows automatically that it is still a pullback in toral
$\bN$--spectra.

First we recall the diagram of $\torus$--spectra from
\cite{tnqcore} in Subsection \ref{subsec:formalTdiagram},
with examples in Subsection \ref{subsec:Tdiagramexamples}.
We then introduce the appropriate general coinduction-type
construction  in Section \ref{sec:equivcoind}.
This is applied in
Section \ref{sec:theNequivcube} to particular equivariant diagrams to
construct the cubical diagram in $\bN$--spectra that we require.

\subsection{The characters}
We start by introducing the main characters used to build a sphere as a homotopy limit of a cubical diagram of $\torus$--spectra.

\begin{defn}
For a connected subgroup $K$ of a torus $\torus$, denote by $\mcF /K$ the family
of all finite subgroups of $\torus /K$ (these are in bijective
correspondence to subgroups of $\torus$ with identity component
$K$). We then consider the $\torus/K$--spectrum  $D\efkp:=F(E\mcF/K_+,\bS)$, the functional dual of the universal space for that family $E\mcF/K_+$ in $\torus/K$--spectra.
\end{defn}

We will be using   maps to relate various commutative ring spectra
$D\efkp$ as $K$ varies. Indeed, $D\efkp$ is a commutative ring
$\torus/K$--spectrum and  if $L\subseteq K$ there is a map
\[
\infl_{\torus/K}^{\torus/L} D\efkp \lra D\eflp
\]
of ring $\torus/L$--spectra \cite[Subsection 6.A]{tnqcore}.
To see where this comes from, we observe that its adjunct
$$\eflp \sm \infl_{\torus/K}^{\torus/L}D\efkp \lra \bS$$
is obtained by composing the $\torus/L$--map $\eflp \lra \efkp$ with
evaluation.

If we have any decreasing sequence
$$G = H_0\supseteq H_1 \supseteq \cdots \supseteq
H_{r-1}\supseteq H_r=1$$
of connected subgroups with $\codim (H_i)=i$, then, omitting notation for inflation,
we have a sequence of
maps of ring $\torus$--spectra whose first term is $\bS$ and whose last term is $D\efp$.
\[
\xymatrix{
\bS=DE(\mcF /\torus)_+ \ar[r] &
D(E\mcF /H_1)_+ \ar[r] &
\ldots \ar[r] &
D(E\mcF /H_{r-1})_+ \ar[r] &
D(E\mcF /1)_+
}
\]

For $K$ a connected subgroup of $\torus$, we define
\[
\siftyV{K}=\bigcup_{U^K=0}S^U, 
\]
where $U$ runs over finite dimensional subrepresentations of $\mcU$ 
(our chosen complete $\torus$--universe) such that $U^K=0$.
If $K \subseteq H$, we see that $V^H \subseteq V^K$, so there is a map
$\siftyV{K} \to \siftyV{H}$.

\subsection{The formal diagram in \texorpdfstring{$\torus$}{T}-spectra}
\label{subsec:formalTdiagram}

In this article we use a cubical diagram which is known as the formal cube
$C_f$ in \cite{tnqcore}. To avoid confusion we use slightly different
notation to that in \cite{tnqcore}, and name  the vertices $(b_0,
\ldots, b_r)$ with $b_i\in \{ 0,1\}$.

First we take $\bfd =(d_0, \ldots , d_s)$ to be the dimension vector
of $\bfb$, i.e., the set $\{ j \st b_j=1\}$ arranged in decreasing order. Then the value of $\RRti$ at $\bfb=(b_0, \ldots, b_r)$ is an iterated product as follows, where the subgroups $H_i$ in the products below are connected subgroups of dimension $d_i$.
\begin{multline*}
\RRti (b_0, \ldots , b_r)=
\prod_{\dim (H_0)=d_0} \left[ \siftyV{H_0} \sm
\prod_{\stackrel{H_1\subset H_0}{  \dim (H_1)=d_1}} \Bigg[
\siftyV{H_1} \sm
 \Bigg. \right.
\ldots  \\
\ldots
\sm
\left. \Bigg.\prod_{\stackrel{H_{s-1}\subset H_{s-2}}{ \dim (H_{s-1})=d_{s-1}}} \bigg[ \siftyV{H_{s-1}} \sm
\prod_{\stackrel{H_{s}\subset H_{s-1}}{ \dim (H_{s})=d_{s}}} \siftyV{H_{s}} \sm
D(E\mcF/H_s)_+\bigg] \cdots \Bigg] \right]
\end{multline*}

See Definition \ref{def:Ncube} for the version in $\bN$--spectra.

\subsection{Examples}
\label{subsec:Tdiagramexamples}
The elaborate notation somewhat obscures the simplicity of this construction.

\begin{exmp}{\em (Rank 1)}
In rank 1 the diagram is
$$\diagram
\bS \rto \dto & \siftyV{\torus}\sm DE\mcF/\torus_+\dto \\
D\efp \rto & \siftyV{\torus}\sm DE\mcF+
\enddiagram$$

\end{exmp}

\begin{exmp}{\em (Rank 2)}
It is worth writing the diagram completely in rank 2. The layout of the $\bfb$ vectors is
$$
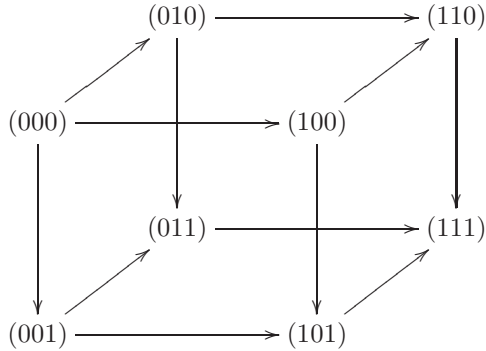
\diagram
&(010)\rrto \ddto&&(110)\ddto\\
(000)\rrto \ddto \urto&& (100) \ddto \urto&\\
&(011)\rrto &&(111)\\
(001)\rrto  \urto&& (101) \urto&
\enddiagram$$
and the diagram of ring spectra is as follows:

\[
\xymatrix@C-1.4cm{
&
\prod_H\siftyV{H}\sm DE\mcF/H_+
\ar[rr] \ar[dd] &&
\siftyV{\torus}\sm \prod_H\siftyV{H}\sm DE\mcF/H_+
\ar[dd]\\
\bS
\ar[rr] \ar[dd] \ar[ur] &&
\siftyV{\torus} \sm DE\mcF/\torus_+
\ar[dd] \ar[ur] \\
&\prod_H\siftyV{H}\sm D\efp
\ar[rr] &&
\siftyV{\torus}\sm \prod_H\siftyV{H}\sm D\efp\\
D\efp
\ar[rr] \ar[ur] &&
\siftyV{\torus} \sm D\efp
\ar[ur]
}
\]

\end{exmp}

The only result we need from \cite{tnqcore} is as follows.

\begin{cor} \cite[Corollary 2.2]{tnqcore}
\label{cor:SisPCfpb}
The cubical diagram $\RRti$ is a homotopy pullback, which is to say that
$\bS$ is the homotopy pullback of the cube diagram $\RRti$
\[
\bS \simeq \underset{v\in \PCf}{\holim} \RRti (v).
\]
Here $PC_f$ denotes the punctured cube, where the initial
vertex has been removed. \qqed
\end{cor}

\begin{rem} We note that all values of the $\PCf$--diagram $\RRti$ are genuinely commutative rational $\torus$--spectra by \cite{CouniversalCommutative_greenlees}.
\end{rem}

\subsection{Diagrams }
From the $\PCf$--diagram $\RRti$ of
commutative ring $\torus$--spectra we may form a number of other diagrams.
Recall that the $\torus$--fixed points of a commutative ring $\torus$--spectra
is a commutative ring spectrum.

\begin{defn}
\label{defn:PCfdiagrams}
From the $\PCf$--diagram $\RRti$ of commutative ring $\torus$--spectra we
form
\begin{enumerate}
\item the $\PCf$ diagram $\RR =\RRti^\torus$ of commutative ring
  spectra
\item  the $\PCf$ diagram of commutative DGAs obtained from
  $\RR$ using the theorem of Richter and Shipley \cite{richtershipley}
	that the category of commutative
	$H\bQ$--algebras is equivalent to commutative DGAs over $\bQ$ (see Section
  \ref{sec:spectratoDGAsformal})
\item the $\PCf$ diagram $\pi^\torus_*(\RRti)=\pi_*(\RR)=\RRa$
  of graded rings.
\end{enumerate}

\end{defn}

In \cite{tnqcore} it is shown that suitable model categories of
modules over these diagrams are all Quillen equivalent.
This is the path from topology to algebra that we will replicate for
$\bN$--spectra. 
Once we are working in the algebraic context of $\RRa$--modules, the task is to 
simplify the algebraic description to
the toral model $d \acal_a^f(\bN, toral)$.

\section{Equivariant diagrams and coinduction from \texorpdfstring{$\torus$}{T} to \texorpdfstring{$\bN$}{N}}
\label{sec:equivcoind}

We now wish to lift the diagram $\RRti$ to
a diagram of $\bN$--spectra which we will call $\RRttop$. It is clear
what we need to do: the Weyl group $W$ acts on the closed subgroups of
$\torus$, preserving their dimension, and we should group together the
subgroups in each $W$-orbit.  In other
words, if $K=K_1, \ldots, K_w$ make a $W$ orbit, with $W_K$ the
isotropy group of $K$, we use the $\torus$--equivalence
\[
F_{W_K}[W, \siftyV{K})\simeq \prod_i\siftyV{K_i}
\]
to replace the product by the coinduced spectrum, where the bracket $[$ is used to indicate
the addition of a basepoint to the domain.

To implement this idea we need to specify an appropriate framework,
and establish it has the required properties. This is straightforward,
but a number of things must be made explicit. We begin with a space level construction, which we will then extend to
spectra levelwise.

\subsection{\texorpdfstring{$W$}{W}--diagrams}

The construction is a slightly elaborated version of a simple and
familiar one we describe here.

If $A$ is a $W$-set we may consider diagrams $D: A\lra \bC$ in a category $\bC$.
For $w\in W$ we may form the pulled back diagram $w^*D$ defined  by $(w^*D)(a)=D(wa)$,
evidently $e^*D=D$ and one quickly checks that $w^*v^*D=(vw)^*D$.
(In \cite{AGtoral} a right action was used on the set of subgroups so the dictionary relating the
two notations is  $w^*D=(w^{-1})_*D$).

\begin{defn}
A  \textbf{$W$--equivariant $A$--diagram} in $\bC$ is
a diagram $D$ equipped with maps $w_m: D\lra (w^{-1})^*D$ which compose in the sense that
$e_m$ is the identity and $v_mw_m=(vw)_m$.
\end{defn}

From this we can form the space  of sections $\Gamma (A,D)=\prod_{a\in
  A} D(a)$, which is a product of objects of $\bC$ with $W$ permuting
the coordinates.

We need an analogous construction when the diagram $D$ does not actually land
in the category $\C$, but so that the shape of the diagram lets us
show the  `space of sections' $\Gamma (A;D)$  takes values in $\C$. The reader should think of a case of a diagram of $\torus$--spaces (spectra) with a $W$--action which we want to view as a diagram of $\bN$--spaces (spectra), by collecting orbits together. We make this example precise below.

\subsection{\texorpdfstring{$\bN$}{N}--spaces over \texorpdfstring{$A$}{A}}
In our general context we have an extension
$$1\lra \torus \lra \bN \stackrel{p}\lra W\lra 1. $$
For any subgroup $K$ of $W$ we may consider the subgroup
$\Kt:=p^{-1}(K)$ of $\bN$, so that in particular, $\Wt=\bN$.

\begin{defn}
For  a $W$--set $A$, an \textbf{$\bN$--space over $A$} is
\begin{itemize}
\item For each $n \in \bN$,  a map $n_m: D(a) \lra D(p(n)a)$ of
  spaces which is equivariant over the group homomorphism 
  $c_{n^{-1}}: \Wt_a \lra \Wt_{p(n)a}$
  (where $c_h$ is the left conjugation map $c_h(g)=hgh^{-1}$).
\item The map $e_m$ is the identity and the maps are transitive in
  that $n_m n'_m=(n n')_m$.
\end{itemize}
\end{defn}

The definition implies that for each $a \in A$, $D(a)$ is a $\Wt_a$--space.
It is straightforward to define a suitable $\bN$--space
of sections.

\begin{defn}
Given an $\bN$--space $D$ over $A$, we define the \textbf{space of sections} to be the product
\[
\Gamma (A, D):=\prod_{a\in A} D(a).
\]
\end{defn}

\begin{lem}\label{lem:spacesection}
The space of sections has the properties

\begin{description}
\item[$\Gamma 0$] $\Gamma (A,D)$ is an $\bN$--space,
\item[$\Gamma 1$] For any $W$-sets $A_i$, the natural map
 $$\Gamma (\coprod_i A_i, D)\stackrel{\cong}\lra \prod_i \Gamma (A_i,
 D)$$
is a homeomorphism,
\item[$\Gamma 2$]  on a $W$--orbit, this is naturally isomorphic to coinduction
\[
\Gamma (W/K, D)\cong F_{\Kt}(\bN_+,  D(eK)).
\]
\end{description}
\qqed
\end{lem}

We note explicitly that
the individual values $D(a)$ are not $\bN$--spaces.
We could define $\Gamma (A,D)$ in terms of coinduction, but this would require a choice of
decomposition of $A$ into $W$--orbits and then a proof that the result is independent of the choice.
Instead, we define sections in terms of the product
$\prod_{a\in A} D(a)$ as this does not involve any choices. However, the easiest way
to see that it is an $\bN$--space is to use
$\Gamma 1$ and $\Gamma 2$ of the above.

\subsection{Naturality of coinduction}
We discuss some of the formal properties of coinduction,
which we can then apply to our sections construction.

We observe that if $\fX$ is an $(H,G)$--bispace (i.e., it
has commuting left $H$--action and right $G$--action) and $Z$ is an $H$--space then $F_H[\fX, Z)$ is
a $G$--space.
This has the properties:
\begin{itemize}
\item If $\fX=G$ we obtain the coinduction from $H$ to $G$.
\item Given  an $(H,G)$--bispace $\fX$, an $(H',G')$--bispace $\fX'$,
and an $H'$--space $Y'$.

 Any map  $\theta : \fX\lra \fX'$ of bispaces over the
group homomorphism $(\beta, \alpha): (H,G)\lra (H', G')$ induces a map
$$\theta^*: \alpha^* F_{H'}[\fX', Y')\lra  F_{H}[\fX, \beta^*Y')$$
of $G$--spaces. This is contravariantly functorial in $\theta$.
\item If $\fX=\coprod_i \fX_i$ as $(H,G)$--bispaces then the natural map
$$F_H[\fX, Z)\stackrel{\cong}\lra \prod_i F_H[\fX_i, Z)$$
is an isomorphism of $G$--spaces.
\end{itemize}

A number of  special cases follow:
\begin{itemize}
\item If $K\subseteq G$ then restricting the $G$--space $F_H[\fX, Z)$ to a $K$--space is given by
  restricting the right action of $G$ on $\fX$ to $K$.
\item If $\theta : G' \lra G$ is a group homomorphism then
\[
\theta^*F_H[\fX, Z)\cong F_H[(\id_H,\theta)^*\fX , Z).
\]
\item If $H$ is of finite index in $G$ and  we write
  $G=\coprod_iH\gamma_i$ then
\[
F_H[G,Y)\cong \prod_i F_H(H\gamma_i,Y)
\]
and we have a homeomorphism of $H^{\gamma}$--spaces
\[
F_H[H\gamma, Y) \cong c_{\gamma^{-1}}^*Y,
\]
where $c_{\gamma^{-1}}: H^{\gamma}\lra H$ is conjugation.
If we write $[\gamma^{-1}]y$ for the point of $c_{\gamma^{-1}}^*Y$ corresponding
to $y$, it is  given by taking  $[\gamma^{-1}]y$ to the $H$--map $h\gamma \longmapsto
h\gamma [\gamma^{-1}]y= hy$.
\end{itemize}

In particular, if $W_a$ a subgroup of $W$ then $\Wt_a$ is of finite
index in $\Wt$ and we can write
$$\Wt=\coprod_i\Wt_ab_i$$
as left $\Wt_a$--spaces. We note that $\Wt_ab$ is a 
$(\Wt_a,\Wt_a^b)$--bispace (where $\Wt_a^b=b^{-1}\Wt_ab$) and we have an isomorphism of left
$\Wt_a^b$--spaces:
$$F_{\Wt_a}[\Wt_a b, Z)\cong b^*Z, $$
so that for any subgroup $K\subseteq \bigcap_i \Wt_a^{b_i}$, we have
an isomorphism
$$F_{\Wt_a}[\Wt, Z)\cong \prod_ib_i^*Z$$
of $K$--spaces.

We can apply the above properties to the case of an $\bN$--space over $A$
for some transitive $W$--set $A \cong W/W_a \cong W/W_b$ and see that the choice of
orbit representatives is usually unimportant.

\begin{cor} \label{lem:otherfibres}
Suppose that $p(n)a=b$ then $D(b)=n^*D(a)$
and the element $n$ gives an isomorphism
\[
F_{\Wt_a}[\Wt, D(a)) \cong F_{\Wt_b}[\Wt, D(b))
\]
of $\Wt$--spaces.
\end{cor}
\begin{proof}
We define $\theta : \Wt \lra \Wt$ by $\theta (g)=ng$ which is a map of
bispaces over the map $(c_n, id):(\Wt_a, \Wt)\lra (\Wt_b, \Wt)$. We may then apply
the naturality property since $n^*D(a)=D(b)$.
\end{proof}

Note that the comparison map for two different choices
depends on $n$ and not just $p(n)$.

\subsection{Homotopy theory}
The spaces $\fX$ that we need to consider (in this paper at least!)
are disjoint unions of copies of $\bN$, but viewed as $(\Kt,\bN)$--bispaces for various subgroups $K$ of $W$.

The category of $(H,G)$--bispaces is equivalent to the category of
$H\times G$--spaces.
In our case, the relevant cells are the
$G$--free cells,  $(H \times G)/K$, so that $K \cap (1\times G)=1$.
Hence we use the equivariant Serre model structure
where weak equivalences and fibrations are the maps that are weak equivalences and fibrations of spaces after taking $K$--fixed points for all subgroups $K$ of $H\times G$ such that  $K \cap (1\times G)=1$.

We then wish to show that if $\fX$ is a cofibrant $(H,G)$--bispace then the functor
$F_H[\fX, \cdot)$ is well behaved.
\begin{lem}
If $\fX$ is a cofibrant $(H,G)$--bispace then the functor
$F_H[\fX, \cdot)$ is a right Quillen functor from $H$--spaces to
$G$--spaces with left adjoint $\fX\times_G(\cdot)$. In particular,
$F_H[\fX, \cdot)$ takes fibrant objects to fibrant objects.
\end{lem}

\begin{proof}
The functor $\fX\times_G (\cdot)$ preserves generating (acyclic) cofibrations, since
it takes  $G$--CW complexes
to spaces admitting the structure of $H$--CW complexes.
\end{proof}

When $\fX$ is the $(H,G)$--bispace $G$, the functor $\fX\times_G (\cdot)$ is precisely
the forgetful functor from $G$--spaces to $H$--spaces.

We also want to be able to change $\fX$. We define a map $\theta : \fX \lra \fX'$ over $(\beta, \alpha)$
to be a weak equivalence if $\fX^M \to ((\beta, \alpha)^* \fX')^M$ is a weak equivalence of spaces for all
subgroups $M$ of $H \times G$ with $M \cap (1\times G)=1$.

\begin{cor}
A weak equivalence $\theta : \fX \lra \fX'$ over $(\beta, \alpha)$
between cofibrant objects induces an equivalence of the two
constructions: with notation as above
$$\theta^*: \alpha^* F_{H'}[\fX', Y')\stackrel{\simeq}\lra  F_{H}[\fX, \beta^*Y')$$
is a Serre weak equivalence of  $G$--spaces. \qqed
\end{cor}

\subsection{\texorpdfstring{$\bN$}{N}--spectra over \texorpdfstring{$A$}{A}}\label{sec:diagramspectra}
All of the above needs to be repeated for diagrams of {\em spectra}.

\begin{defn}
If $A$ is a $W$--set, then $D$, an \textbf{$\bN$--spectrum over $A$}, is a collection
\[
\{ D(a) \in \Wt_a \spO \mid a \in A \}
\]
where $\Wt_a \spO$ is indexed over a universe
obtained by restricting from a complete $\bN$--universe,
along with the data:
\begin{itemize}
\item For each $n \in \bN$,  a map $n_m: D(a) \lra D(p(n)a)$ of
  $\Wt_a$--spectra over the group homomorphism $c_{n^{-1}}: \Wt_a \lra \Wt_{p(n)a}$
(where $c_h$ is the left conjugation map $c_h(g)=hgh^{-1}$).
\item The map $e_m$ is the identity and the maps are transitive in
that $n_m n'_m=(n n')_m$ for all $n, n' \in \bN$.
\end{itemize}
\end{defn}

The idea is to define sections as before by the formula
\[
\Gamma (A, D)(V):=\prod_{a\in A} D(a)(V).
\]
To see that the above construction gives an $\bN$--spectrum, we want to recognise the sections as a product of
coinduced $\bN$--spectra. If $A$ is a transitive $\bN$--space with associated conjugacy class $(H)$
and we choose a representative subgroup $H$, the construction
$F_H[\bN,\cdot)$ extends to orthogonal spectra.  If $V$ is a
representation of $\bN$ we may view it as a representation of $H$ by
restriction and define
$$F_H[\bN, D)(V):= F_H[\bN , D(V)). $$
The structure $\bN$--map
\[
S^{V'}\sm  F_H[\bN , D(V'')) \lra F_H[\bN,D(V'\oplus V''))
\]
has adjunct the $H$--map
\[
S^{V'}\sm F_H[\bN, D(V''))
\lra
S^{V'}\sm F_H[H, D(V''))
=
S^{V'}\sm D(V'')
\lra
D(V'\oplus V'').
\]
The results of this construction for various representatives of the
conjugacy class are related by the maps $n_m$ as required.

The construction makes it obvious that  the objectwise suspension spectrum functor from $\bN$--spaces over a set to
$\bN$--spectra over a set, gives a natural equivalence
\[
F_H[\bN, \Sigma^{\infty}D) \cong \Sigma^{\infty} F_H[\bN, D).
\]
It follows that we have the spectrum level analogue of Lemma \ref{lem:spacesection}.
We also want to consider smash products, see
Section \ref{subsec:smashsections} for a brief discussion.

\begin{lem}\label{lem:spectrasection}
The spectrum of sections has the properties:
\begin{description}
\item[$\Gamma 0$] $\Gamma (A,D)$ is an $\bN$--spectrum,
\item[$\Gamma 1$] the natural map
 $$\Gamma (\coprod_i A_i, D)\stackrel{\cong}\lra \prod_i \Gamma (A_i,
 D)$$
is a homeomorphism,
\item[$\Gamma 2$]  on a $W$--orbit, this is naturally isomorphic to coinduction
\[
\Gamma (W/K, D)\cong F_{\Kt}[\bN,  D(eK)),
\]
\item [$\Gamma 3$] if $D$ is a diagram of commutative rings
  then $\Gamma ( A; D)$ is a commutative ring $\bN$--spectrum. \qqed
\end{description}
\end{lem}

Next we observe that the definition is homotopically well behaved. We need to establish
analogues of the results for spaces.
It is well known \cite[Section V.2]{mm02} that for orthogonal spectra the functor $F_H[\bN,\cdot)$
is a right Quillen functor with left adjoint restriction.

We give the category of $\bN$--spectra over $A$
the projective model structure, so that weak equivalences and
fibrations are pointwise weak equivalences and fibrations of $\torus$--spectra.
Note that it suffices to require these conditions
at only one point in each orbit.
Furthermore, the cofibrant objects are in particular pointwise cofibrant.

\begin{lem}
\label{lem:Grq}
The sections functor $\Gamma$ is a right Quillen functor when
the category of $\bN$--spectra over $A$ is equipped with the projective model structure.
\end{lem}
\begin{proof}
The sections functor is a product of functors to $\bN$--spectra of the form
$F_{\Kt}[\bN, \cdot)$, this functor takes weak equivalences of $H$--spectra to
weak equivalences of $\bN$--spectra. Hence $\Gamma$ preserves all weak equivalences.

Similarly the functor $F_{\Kt}[\bN, \cdot)$ takes fibrations of
$H$--spectra to fibrations of $\bN$--spectra.
\end{proof}

\subsection{Duals and smash products}\label{subsec:smashsections}
For $X$, an $\bN$--spectrum over $A$, we define a dual $DX$
that is also an $\bN$--spectrum over $A$.
We define $DX$ by $D(X(a))=F(X(a), \fibrep \bbS)$,
using a fibrant replacement of the sphere spectrum
without further comment.
The structure maps of $X$ induce the structure maps of $DX$ as below.
\[
X(a)\stackrel{n_m}\lra X(p(n)a)
\quad
\quad
F(X(a), \fibrep \bbS )\stackrel{n_m^*}\longleftarrow F(X(p(n)a), \fibrep \bbS).
\]
Furthermore, $X(p(n)a)\cong p(n)^*X(a)$, so
\[
F(X(p(n)a), \fibrep \bbS) =F(p(n)^*X(a), \fibrep \bbS)=p(n)^*F(X(a), \fibrep \bbS)
\]
and taking $n^{-1}_m=n_m^*$ we obtain the requisite structure maps.
If $X$ is cofibrant, then each $X(a)$ is cofibrant and hence $DX$ is
a homotopically meaningful construction.

More straightforwardly, there is a smash product $X \sm X'$  of
$\bN$--spectra over $A$ defined
by taking smash products pointwise:
\[
(X\sm X')(a)=X(a)\sm X'(a).
\]

The following result follows immediately from the definitions and the fact that coinduction
preserves commutative rings.
\begin{lem}
If $X$ is an $\bN$--spectrum over $A$ consisting of commutative ring spectra
then the $\bN$--spectrum $\Gamma (A; X)$ is a
commutative ring spectrum. \qed
\end{lem}

Since the positive stable model structure on $\bN$--spectra over $A$
is defined objectwise, it extends to the level of
commutative rings.

\begin{lem}\label{lem:commdiagram}
There is a model structure of commutative ring objects in
$\bN$--spectra over $A$.
\end{lem}
\begin{proof}
Let $A$ be a transitive $W$--set, $A=W/K$. Then
if $D$ is an $\bN$--spectrum over $A$, $D(eK)$ is a
$\Kt$--spectrum.
Consider the positive stable model structure on
$\Kt$--spectra. This right lifts to the projective model structure on
$\bN$--spectra over $A$, with generating cofibrations and acyclic cofibrations
given by the $W$--equivariant coproduct over $A$ of the generating
cofibrations and acyclic cofibrations of the positive stable model structure on
$\Kt$--spectra.
We can then lift this model structure to the level of commutative rings
noting that the free commutative ring functor sends
coproducts to smash products.

A general $W$--set $A$ is a coproduct of transitive sets $A_i$, for 
$i$ in some indexing set $I$. 
The category of $\bN$--spectra over $A$ is then the product of 
the categories of $\bN$--spectra over $A_i$, $i \in I$. 
We use the product model structure on $\bN$--spectra over $A$
and on commutative ring objects in $\bN$--spectra over $A$.
\end{proof}

\subsection{\texorpdfstring{$\torus$}{T}--fixed points}
If $D$ is an $\bN$--spectrum over $A$, then we can understand $\torus$--fixed
points of the spectrum of sections directly.
\begin{prop}
\label{prop:GammaTfixed}
There is an equivalence
$$\Gamma (A, D)^\torus \simeq \Gamma (A, D^\torus)$$
\end{prop}
\begin{proof}
Since fixed points commute with products, we need only deal with the
case of a transitive $W$--set $A=W/W_a$, and then
\[
F_{\Wt_a}[\Wt , D)^\torus  \simeq F_{W_a}[W, D(a)^\torus)\simeq F[W/W_a, D(a)^\torus). \qedhere
\]
\end{proof}

\section{The \texorpdfstring{$\bN$}{N}--equivariant pullback cube}
\label{sec:theNequivcube}
Where possible we start all constructions at the space level and then
take suspension spectra.

First we consider an action
of $W$ on a set $A$, and for each $a$ we can find a $\Wt_a$--space $D(a)$ so
that for $n\in \bN$ we have an isomorphism
$$n_m:D(p(n)a)\stackrel{\cong}\lra n^*D(a)$$
where $e_m=id$ and $n'_mn''_m=(n'n'')_m$.

\subsection{First examples}

\begin{exmp}
We may consider the set $A=\mcF$ of finite subgroups of $\torus$.
The group $W$ acts on $\mcF$ and a chosen subgroup $F$ is
fixed  by $W_F$. The normalizer of the  subgroup $F$  is
$\Wt_F=p^{-1}(W_F)$. If we use a functorial construction (such as the
bar construction) to make universal
spaces $E\mathcal{H}$, and we define $\elr{H}$ as the mapping cone
$$E[\subset H]_+\lra E[\subseteq H]_+\lra \elr{H}$$
in  $\Wt_F$--spaces then
$$c_{n^{-1}}^*\elr{H}=\elr{H^{n}}$$
and the maps $n_m$ can all be taken to be the identity.
Accordingly this gives an  $\bN$--space  $\elr{\bullet}$ over $\mcF$.
\end{exmp}

Building on this example we may construct an  $\bN$--spectrum over $\mcF$ of functional dual $D\elr{\bullet}$.
\begin{exmp}
Starting with the previous example, $\elr{\bullet}$, taking suspension
spectra gives an $\bN$--spectrum over $\mcF$. We may then take the
fibrewise dual to obtain $D\elr{\bullet}$. More precisely, the value
at $H$ is the $\Wt_H$--spectrum $D\elr{H}$ and we still have
$c_{n^{-1}}^*D\elr{H}=D\elr{H^{n}}$.

Taking sections we obtain the $\bN$--spectrum
\[
\Gamma (\mcF , D\elr{\bullet})\simeq D\efp.
\]
\end{exmp}

\begin{exmp}\label{exmp:siftyV}
We may consider the set $A=\Sc$ of connected closed subgroups of $\torus$. Once
again, if $K$ is a subgroup of the torus, it is fixed
by $W_K\subseteq W$. Choosing a complete $\bN$--universe $\mcU$  we may form  the $\Wt_K$--space
$$\siftyV{K}=\bigcup_{W^K=0}S^W, $$
where $W$ runs over finite dimensional $\Wt_K$--subrepresentations of $\mcU$. A little representation theory verifies that the
geometric isotropy consists precisely of the subgroups containing
$K$. This explains the notation. Pulling back along $n$ induces a
bijection $n_m: \{ W\st W^K=0\} \lra \{ W'\st W^{K^n}=0\}$, and hence
$n^*\siftyV{K}=\siftyV{K^n}$.
This therefore gives an $\bN$--space $\siftyV{\bullet}$ over $\Sc$.

We can of course restrict attention to  the set $\Sc_d$ subgroups of
dimension $d$ and   we then have a $\torus$--equivalence
\[
\Gamma (\Sc_d, \sifty{\bullet}) \simeq_\torus \prod_{K\in \Sc_d}\siftyV{K}.
\]

The map $S^0 \to S^{\infty V(K)}$ induces a map of
$\bN$--spectrum over $\Sc$ (or $\Sc_d$) from the constant diagram at
$S^0$ to $\siftyV{\bullet}$.

This $\bN$--spectrum over $\Sc$ is an $E_\infty^G$--ring spectrum by
\cite[Corollary 4.8]{CouniversalCommutative_greenlees}. Since
$E_\infty^G$--ring spectra are the commutative monoids in orthogonal
$G$--spectra, we consider $\siftyV{\bullet}$
as a commutative ring spectrum.
\end{exmp}

\begin{exmp}
We may again consider the $W$--set $A=\Sc$ of connected closed
subgroups of a torus $\torus$. Since the subgroup $K$ is fixed by
$W_K$, the set $\mcF/K$  of subgroups of $\torus$ with finite image in $\torus/K$
is a family of subgroups of $\Wt_K$. Evidently conjugation by $n$
gives a bijection $\mcF/K \stackrel{\cong}\lra \mcF /K^n$, and hence
$c_{n^{-1}}^*E\mcF/K_+=E\mcF/K^n_+$.
Accordingly the assignment of $E\mcF/K_+$ to $K$ defines an
$\bN$--space  $E\mcF/{(\bullet)}_+$ over $\Sc$.

Taking functional duals we obtain the $\bN$--spectrum $DE\mcF/(\bullet)_+$ over $\Sc$.
The cocommutative diagonal of the space $E\mcF/K_+$ induces a
commutative ring structure on the dual. This ring structure is compatible
with the structure maps $c_{n^{-1}}^*DE\mcF/K_+=DE\mcF/K^n_+$.
It follows that the $\bN$--spectrum $DE\mcF/{(\bullet)}_+$ over $\Sc$
is a diagram of commutative rings.
\end{exmp}

Our main diagram of $\bN$--spectra (given in Definition \ref{def:Ncube}),
is built from diagrams of the form $DE\mcF/(\bullet)_+$
by repeatedly applying the sections functor and localising
at the diagram $\siftyV{\bullet}$.

\subsection{Posets}
We now want to consider a partially ordered set $\Sigma$ with an action of $W$, and a dimension function so that if $s< t$ then $\dim(s)< \dim (t)$.
We might want to insist that the poset has a maximal element $G$ and the dimension is determined by the length of chain to $G$, but in any case we
want the $W$--action to preserve dimension.

This enables us to construct numerous $W$--sets, starting with the set $\Sigma_i$ of elements of dimension $i$. We may consider the
poset $\Sigma'$ of flags $(H_0>H_1>\cdots >H_s)$ in $\Sigma$. This is again a $W$--set, as is the set $\Sigma_s$ of flags of length $s$, and the set
$\Sigma_{(d_0>d_1>\cdots >d_s)}$ of flags of length $s$ with a specified dimension vector $\dim (H_i)=d_i$.

Furthermore,    if $K$ is an object of $\Sigma$  and $e\leq \dim K$ we have the $W$--set
\[
\Sigma_{e\leq K}=\{ L\st L \leq K \mbox{ and } \dim(L)=e\}
\]
and we note that this is a $W_K$--set.

If $D$ is a $\Wt_K$--spectrum over $\Sigma_{e\leq K}$ then,
applying sections gives a $\Wt_K$--spectrum $\Gamma (\Sigma_{e\leq K}, D)$.
We now let $K$ run through elements of dimension $d$.  We require in
addition  that the first section spectrum
satisfies the functoriality as $K$ varies, namely that there are isomorphisms
\[
n_m: \Gamma (\Sigma_{e\leq K}, D) \stackrel{\cong}\lra n^*\Gamma
(\Sigma_{e\leq K^{n^{-1}}}, D) \tag*{($\dagger$)}
\]
which compose functorially.  We may then iterate the
construction.   Indeed, if $e \leq d\leq \dim (H)$, we consider the functor
$$\Sigma_{d\leq H}\ni K \mapsto \Gamma (\Sigma_{e\leq K}, D), $$
and note that it is a $\Wt_H$--spectum over $\Sigma_{d \leq H}$.

One way to ensure the additional property ($\dagger$) is to require that
 $D$ is defined and $\Wt_H$--equivariant on all dimension $e$ elements
$\leq H$.

To display the iterated sections compactly we use
the calculus--style notation
\[
\Gamma (A\!\ni \!a \; ;\;  D(a))=\Gamma(A,D)
\]
to display the names of elements.

\begin{lem}\label{lem:combinesections}
If $D$ is a $\Wt_H$--spectrum over $\Sigma_{e\leq H}$,
with $e\leq d \leq \dim H$ then there is a natural diagonal map
\[
\Gamma (\Sigma_{e\leq H}\ni L; D(L))\lra
\Gamma (\Sigma_{d\leq H}\ni K ; \Gamma (\Sigma_{e\leq K}\ni L;
D(L))),
\]
arising since in the domain there is just one factor for each
$L\subset H$ with $\dim (L)=e$, whereas on the right there is one for
each chain $L\subset K \subset H$ with $\dim (L)=e, \dim (K)=d$. \qqed
\end{lem}

\subsection{Dimensional coefficient systems}
We iterate the sections construction,
smashing each stage with $\siftyV{\bullet}$
from Example \ref{exmp:siftyV} and
taking fibrant replacements in commutative ring spectra.
As taking sections, smashing with commutative rings
and taking fibrant replacements preserve
commutative rings, it follows that each stage in this iteration
will be a fibrant commutative
ring spectrum.

It will be clear from the definitions that the final
stage will be a fibrant commutative ring $\bN$--spectrum
which forgets to the same $\torus$--spectrum as in
Section \ref{subsec:formalTdiagram}.

Suppose that $E$ is a commutative ring
$\Wt_K$--spectrum over $\Sigma_{e\leq K}$.
Then we can form a new diagram of commutative ring spectra
over $\Sigma_{e\leq K}$ which at place $L$ takes value
$\siftyV{L} \sm E(L)$.
We define $\bL_{V} E$ to be the functorial fibrant replacement of
this object in the category of
$\Wt_K$--spectrum over $\Sigma_{e\leq K}$ of
Lemma \ref{lem:commdiagram}. We write
$(\bL_{V} E)(L) = \bL_{V(L)}E(L)$ to help remind us that we want
$\siftyV{L}$ at place $L$.

For a $K$ of dimension $d$, we have seen that
\[
\Gamma (\Sigma_{e\leq K}\ni L, \bL_{V(L)} E(L))
\]
is also a commutative ring object. Allowing $K$ to vary, we have
commutative ring objects over the diagram $\Sigma_{d\leq H}$.
\[
\Gamma (\Sigma_{e\leq K}\ni L, \bL_{V(L)} E(L))
\quad \textrm{and} \quad
\bL_{V(K)} \Gamma (\Sigma_{e\leq K}\ni L, \bL_{V(L)} E(L))
\]
We may then define
\[
\Gamma (\Sigma_{e\leq d\leq H}; \bL, E):=
\Gamma (\Sigma_{d\leq H}\ni K;  \bL_{V(K)}
	\Gamma (\Sigma_{e\leq K}\ni L, \bL_{V(L)} E(L)) ).
\]

\begin{defn}
Given a dimension vector  $\bfd=(d_0, d_1, \ldots , d_s)$, and $D$
an $\bN$--spectrum over $\Sigma$,
we may define the iterated localized sections of $D$ at $\bfd$ as
\begin{multline*}
\Gamma (\Sigma_{\bfd}; \bL, D):=
\Gamma(\Sigma_{d_0} \ni H_0;
\bL_{V(H_{0})}
\Gamma(\Sigma_{d_1\subseteq H_0}\ni H_1;
\cdots \\
\cdots  \;
\bL_{V(H_{s-1})}
\Gamma(\Sigma_{d_s\subseteq H_{s-1}}\ni H_s;  \bL_{V(H_s)} D_s(H_s))\cdots ))
\end{multline*}
\end{defn}
We are most interested in the case where $D$ is the $\bN$--spectrum over $\Sigma^c$
given by $D(K):=DE\mcF/K_+$ (recall that $\Sigma^c$ is the poset of
connected subgroups of~$\torus$).

\subsection{Face maps}
For a fixed $D$, we want to turn $\Gamma (\Sigma_{\bfd}; \bL, D)$ into a diagram
over the poset of flags in $\Sigma^c$.
We need to explain how inclusions of dimensional faces $\bfd \lra \bfe$
induce maps of sections
\[
\Gamma(\Sigma_{\bfd}; \bL, D)\lra \Gamma(\Sigma_{\bfe}; \bL, D).
\]
The following two constructions will cover what we need.

\begin{construction}
\label{con:A}
Let $E$ be an $\bN$--spectrum over $\Sigma$ and
$\bfd$ be a dimension vector formed from $\bfe$ by omitting the $i$
vertex and $i$ is not the last vertex.
Using Lemma \ref{lem:combinesections}
and the natural transformation
$\id \to \bL_{ V(\bullet)}$
(coming from the definition of a fibrant replacement and the maps
$S^0 \to \siftyV{\bullet}$)
we have the composite map
\begin{multline*}
\Gamma (\Sigma_{e\leq H}\ni L; E(L))
\lra
\Gamma (\Sigma_{d\leq H}\ni K ; \Gamma (\Sigma_{e\leq K}\ni L; E(L))) \\
\lra
\Gamma (\Sigma_{d\leq H}\ni K ; \bL_{V(K)} \Gamma (\Sigma_{e\leq K}\ni L; E(L))).
\end{multline*}
We note that the first map is a diagonal map.
If $a$ is the dimension preceding $i$ and $b$ is the one after (so that $a>i>b$) then in the domain, if we pick
$H$ of dimension $a$,  each
object $L$ of dimension $b$ contained in $H$ occurs only once, but in the codomain, it occurs once for
each flag $H>K>L$ with $\dim (K)=i$.
\end{construction}

\begin{construction}
\label{con:B}
The second construction will only be used to add on a new final vertex
\[
\bfd = (d_0, d_1, \dots , d_s ) \to (d_0, d_1, \dots , d_s ,d_{s+1}) =\bfe.
\]
In this case we assume that we have a map
of $\bN$--spectra over $\Sigma_{d_{s} \subseteq H_{s-1}}$
\[
\bL_{V(H_s)} D(H_s)
\lra
\bL_{ V(H_s)} \Gamma (\Sigma_{d_{s+1} \subseteq H_s}\ni H_{s+1};
\bL_{ V(H_{s+1})} D(H_{s+1})).
\]
By applying the same sequence of functors to domain and codomain
this gives a map
\[
\Gamma(\Sigma_{\bfd}; \bL, D) \lra \Gamma(\Sigma_{\bfe}; \bL, D).
\]
\end{construction}

\subsection{The \texorpdfstring{$\bN$}{N}--equivariant pullback cube}
We define the cube of interest to the other sections of this paper.
We give the definition first, then verify that this is a cubical diagram of $\bN$--spectra.

For $\bfb$ a point of the cube,
let $\bfd(\bfb) =(d_0, \ldots , d_s)$ to be the dimension vector
of $\bfb$: the set $\{ j \st b_j=1\}$ arranged in decreasing order.

\begin{defn}\label{def:Ncube}
The definition
\[
\RRttop (b_0, \ldots , b_r)=\Gamma (\Sigma^c_{\bfd (\bfb)}; \bL,  DE\mcF/(\bullet)_+)
\]
gives a cubical diagram of commutative ring $\bN$--spectra.
\end{defn}
We know by Lemma \ref{lem:spectrasection}
that each term is a commutative ring $\bN$--spectrum. Moreover,
each term is fibrant, as one operation is to apply a fibrant replacement functor
and the sections functor is a right Quillen functor by Lemma \ref{lem:Grq}.
It is routine to check that this construction lifts the
version for $\torus$--spectra to the level of $\bN$--spectra.

\begin{lem}\label{lem:Ncubeproperties}
Let $i^*$ denote the forgetful from $\bN$--spectra to $\torus$--spectra.
The diagram of $\torus$--spectra $i^* \RRttop$ is objectwise fibrant and is weakly equivalent
to the diagram $\RRti$ of $\torus$--spectra of Section \ref{sec:liftedformalcube}. \qed
\end{lem}

We must also define the maps of this cubical diagram.
Construction \ref{con:A}  supplies all the face maps which do not involve changing the final vertex.
When adding the final vertex, we may suppose $\dim (H_s)=d_s$, and by
Construction \ref{con:B}  it suffices to observe that there is a map
\[
DE(\mcF/H_s)_+ \lra
\Gamma (\Sigma_{d_{s+1} \subseteq H_s}\ni H_{s+1};
\bL_{V(H_{s+1})} DE(\mcF/H_{s+1})_+).
\]
The basic ingredient is the map
\[
\Delta: DE(\mcF/H_s)_+ \lra
\prod_{\stackrel{H_{s+1}\subset H_s}{\dim (H_{s+1})=d_{s+1}}}  DE(\mcF/H_{s+1})_+
\]
whose components are inflations
\[
DE(\mcF/H_s)_+ \lra  DE(\mcF/H_{s+1})_+.
\]
The map $\Delta$ is composed with the natural transformation
$\id \to \bL_{V(\bullet)}$ in each factor.

We also need to check that it does not matter which order we add vertices in. There are essentially three cases, according to how many of the added
pair of vertices are last. We suppose given dimension vectors $\bfd
\subset \bfe \subset \bff$ where $\bfe$ is formed by adding one vertex
to $\bfd$ and  $\bff$ is formed by adding one vertex to $\bfe$. We suppose $a$ is the lowest dimension  (last vertex) of $\bfd$.
In the first case the added vertices are $i,j$ with $i<j<a$; in the second they are $i,j$ with $i<a<j$; in the third they are $i,j$ with $a<i<j$. We therefore
need to consider the commutativity of the diagrams suggested below.
$$\diagram
a\rto \dto &i,a \dto     & a\rto \dto &i,a \dto                & a\rto \dto &a,i \dto\\
j,a\rto &i,j,a                 & a,j\rto        &i,a,j                     &a,j\rto &a,i,j
\enddiagram$$

In the first and second cases, commutation is clear from the unit axiom and the universal property of the product. In the final case, there is the
additional ingredient that the composite of inflation
maps
\[
DE\mcF/H_+\lra DE\mcF/K_+\lra DE\mcF/L_+
\]
is equal to the direct inflation map $DE\mcF/H_+\lra DE\mcF/L_+$ (where $\dim L=a$, $\dim K=b$
and $\dim H=c$), which we deal with in Subsection
\ref{subsec:inflcomp}.

\subsection{Composites of inflation maps}
\label{subsec:inflcomp}
In this subsction we consider the inflation maps in more detail,
observing that with an appropriate set of details we can ensure
actual functoriality (rather than just up to homotopy).

To start with, note that the construction takes the family
$\mcF/K$ of subgroups of $\torus/K$. This is also a family of subgroups of
$W_\bN(K)=N_\bN(K)/K$, and viewing it in this way we may form the
universal $W_\bN(K)$--space $E\mcF/K$ and then the
$W_\bN(K)$--spectrum $DE\mcF/K_+$. We then inflate this to form the
$N_\bN(K)$--spectrum $\infl DE\mcF/K_+$. For most of the paper  we omit the notation for
inflation, but in this subsection the additional notation is necessary
so that we can be clear (the point is that this spectrum is generally inequivalent to
the spectrum $D\infl E\mcF/K_+$ formed by taking the
$N_\bN(K)$--equivariant dual).

Now suppose that we have an inclusion $L\leq K$ of connected
subgroups of $\torus$.  We consider the induced map
\[
\infl D E\mcF/K_+\lra \infl D E\mcF/L_+.
\]
The first thing to say is that the first spectrum is an
$N_\bN(K)$--spectrum and the second is an $N_\bN(L)$--spectrum, so that the
most we can hope is for a map which is $N_\bN(K)\cap
N_\bN(L)$--equivariant. For brevity, we write
\[
N'=N_\bN(K\supseteq L)=N_\bN(K)\cap N_\bN(L)
\]
for this simultaneous normalizer, and note that it is a finite extension of $\torus$.

Next, we note that if we are more careful about the ambient
equivariance in the dualization we have a rational equivalence
\[
\infl D_{W_\bN(K)}E\mcF/K_+
\simeq \infl D_{N'/K}E\mcF/K_+
\]
of $N'$--spectra, and similarly for $L$. On that basis we may omit
subscripts for dualization in other sections.

Now the $N'$--map
$$\infl D_{W_\bN(K)}E\mcF/K_+\lra \infl D_{W_\bN(L)} E\mcF/L_+$$
is defined as follows. Since $L\leq K$ we may inflate from $N'/K$ to $N'/L$ and then inflate
the whole map to $N'$. This is the inflation of a  map
$$\infl D_{N'/K}E\mcF/K_+\lra  D_{N'/L} E\mcF/L_+, $$
which is the adjunct of
\begin{multline*}
E\mcF/L_+\sm \infl D_{N'/K}E\mcF/K_+\lra E\mcF/K_+\sm \infl
D_{N'/K}E\mcF/K_+ \\
=\infl ( E\mcF/K_+\sm D_{N'/K}E\mcF/K_+)\stackrel{\infl(ev)}
\lra  S^0.
\end{multline*}

Now  suppose we have cotoral inclusions and consider $L\leq K\leq H$.
The map
\[
\infl D_{W_\bN(H)}E\mcF/H_+\lra \infl D_{W_\bN(L)}E\mcF/L_+
\]
for the inclusion $L\leq K$ is $N_\bN(K\supseteq L)$--equivariant, but if
we restrict to $N_\bN(H\supseteq K\supseteq L)$, the defining conditions
on the adjunct are satisfied by the adjunct of
$$\infl D_{W_\bN(H)}E\mcF/H_+\lra \infl D_{W_\bN(K)}E\mcF/K_+ \lra \infl
D_{W_\bN(L)}E\mcF/L_+. $$
Accordingly the inflation maps are functorial as required.

\subsection{Passage to cubical diagrams}
The results of Sections \ref{sec:equivcoind} and
\ref{sec:theNequivcube} allow us to move from a localization of $\bN$--spectra
to a cellularization of modules over our cubical diagram of rings.

\begin{defn}\label{def:cells}
We let $\cKS=\{ \bN/L_+ \sm \RRttop  \mid L \leq \torus \}$ denote the cells of
$\RRttop \modin L_{e_{\torus_\bN}S_{\bQ}}(\NspO)$.
\end{defn}
Note that $\{ \bN/L \mid L \leq \torus \}$
is a set of generators for toral--$\bN$--spectra and it forgets to
a set of generators for $\torus$--spectra. Applying the functor
$(-) \sm \RRttop$ to this set gives $\cKS$.
The objects of $\cKS$ are small in the homotopy category of
$\RRttop \modin L_{e_{\torus_\bN}S_{\bQ}}(\NspO)$
by a similar argument to that above
\cite[Proposition 3.2.5]{BGKSso2}.

\begin{prop}
The adjunction $((-) \sm \RRttop, \lim)$ induces
a Quillen equivalence
$$L_{e_{\torus_\bN}S_{\bQ}}(\NspO)\quillen \cKS \cell \RRttop \modin L_{e_{\torus_\bN}S_{\bQ}}(\NspO).$$
\end{prop}
\begin{proof}
As $\RRttop$ forgets to $\RRti$ by Lemma \ref{lem:Ncubeproperties},
and in each category the weak equivalences are defined via a
forgetful functor to a model category constructed from $\torus$--spectra
we may lift the Quillen equivalence of \cite[Corollary 6.3]{tnqcore}
to the level of $\bN$--spectra.
\end{proof}

\section{Passage to torus fixed points}\label{sec:torusfixed}
The purpose of this section is to understand the homotopy groups of the $\torus$--fixed points of
the rings $\RRttop (b_0, \ldots , b_r)$ of Definition \ref{def:Ncube}.

Notice that taking $\torus$--fixed points objectwise to $\RRttop$
will move us from a diagram of
commutative ring $\bN$--spectra to a diagram $\RRtop$ of
commutative ring $W$--spectra, where $W=\bN/\torus$ is a finite group.
Recalling that each term is fibrant, we see that
the answer will be homotopically meaningful. Accordingly, we consider
the diagram
\[
\RRtop (b_0, \ldots , b_r):=\RRttop (b_0, \ldots , b_r)^\torus
\]
of $W$--spectra.  Taking $\torus$--fixed points objectwise gives an adjunction
\[
\adjunction{\Rinf}
{\RRtop \modin L_{e_1S_\bQ}(\WspO)}
{\RRttop \modin L_{e_{\torus_\bN}S_{\bQ}}(\NspO)}
{\Psi^\torus}
\]
by the results of \cite{gsfixed}.

\begin{prop}
The adjunction $(\Rinf , \Psi^\torus)$ induces
a Quillen equivalence
\[
\RRtop \modin L_{e_1S_\bQ}(\WspO)
\quillen
\RRttop \modin L_{e_{\torus_\bN}S_{\bQ}}(\NspO)
\]
and hence an equivalence between the cellularized categories
\[
\cKS^\torus \cell \RRtop \modin L_{e_1S_\bQ}(\WspO)
\quillen
\cKS \cell \RRttop \modin L_{e_{\torus_\bN}S_{\bQ}}(\NspO)
\]
\end{prop}
\begin{proof}
By Lemma \ref{lem:Ncubeproperties} the cube $\RRttop$ forgets to the cube
$\RRti$, so it follows that $\RRtop$ forgets to the cube
$\RR=\RRti^\torus$.
In each category the weak equivalences are defined via a
forgetful functor to a model category constructed from $\torus$--spectra,
we can lift the $\torus$--spectrum equivalence of \cite[Theorem 7.6]{tnqcore}
to the level of $\bN$--spectra.
\end{proof}

The next step is to replace the underlying category $L_{e_1S_\bQ}(\WspO)$
with $W$--objects in $H \bQ$--modules in symmetric spectra:
$H\bQ \leftmod[W]$.
We do so by a series of strong symmetric monoidal Quillen equivalences,
each of which will create a new cube of commutative ring spectra
by applying either the left or right adjoint as appropriate.
For simplicity we will keep the same notation $\RRtop$
for the cube of ring spectra throughout.

The first step is to use  a combination of changing the model
structure and the universe on $W$--spectra gives a Quillen equivalence between
$L_{e_1S_\bQ}(\WspO)$ and $\spO_\bQ[W]$, $W$-objects
in orthogonal spectra.
The key  is to recognise that in both model categories
the weak equivalences are defined by forgetting to non-equivariant spectra.
Thus the equivalence of categories is in this case also a Quillen equivalence.
Moreover it lifts to a Quillen equivalence between
$\RRtop \modin L_{e_1S_\bQ}(\WspO)$ and
$\RRtop \modin \spO_\bQ[W]$.

The second step is the (zig-zag) of Quillen equivalences between
$\spO_\bQ[W]$ and
$H\bQ \leftmod[W]$ given by forgetting to rational
symmetric spectra and applying $(-) \sm H \bQ$
lifts to give a Quillen equivalence between
$\RRtop \modin \spO_\bQ[W]$
and $\RRtop \modin (H\bQ \leftmod[W])$.
We summarise this in the following result.

\begin{prop}
There is a  zig-zag of strong symmetric monoidal Quillen equivalences
\[
\RRtop \modin L_{e_1S_\bQ}(\WspO)
\quillen
\RRtop \modin (H\bQ \leftmod[W]).
\]
This zig-zag induces a
zig-zag of Quillen equivalences
\[
\cKS^\torus \cell \RRtop \modin L_{e_1S_\bQ}(\WspO)
\quillen
\cKS^\torus \cell \RRtop \modin (H\bQ \leftmod[W]).
\]
\end{prop}

\section{Passage to algebra and formality}
\label{sec:spectratoDGAsformal}

We have shown that the category of $\bN$--spectra is equivalent to the
cellularization of modules over a diagram $\RRtop$ of commutative
ring spectra with $W$--action. Recall that this diagram has a shape of a punctured cube.

There are various choices for which category these commutative
rings (i.e. vertices of $\RRtop$) lie in, but we have arranged the rest of the account so
that they are rings in the category $(H\bQ\leftmod)[W]$ of objects of
$H\bQ\leftmod$ with $W$--action. Recall that $W$ is a finite group and
by naturality Shipley's work
\cite{shiHZ} extends to give a symmetric monoidal Quillen equivalence between
$H\bQ\leftmod[W]$ and chain complexes of $\bQ[W]$--modules, denoted $\bQ[W] \leftdmod$.
Moreover, this result implies that a diagram of rational commutative ring
spectra with a $W$--action $\RRtop$ gives rise to a diagram of rational commutative DGAs
$\RRt=\Theta \RRtop$ with $W$--action (due to the functoriality of $\Theta$, which is the derived functor from Shipley's theorem \cite{shiHZ} from topology to algebra).

\begin{prop}\label{prop:alg-qe}
There is a Quillen equivalence
\[
\RRtop \modin (H\bQ\leftmod[W]) \quillen
\RRt \modin \bQ[W] \leftdmod
\]
between the category of module spectra $\RRtopmod$
and the category of differential graded modules $\RRtmod$.
Furthermore, there is a zig-zag of Quillen equivalences
\[
\cKS^\torus \cell \RRtop \modin (H\bQ\leftmod[W]) \quillen
\cKS_t \cell \RRt \modin \bQ[W] \leftdmod
\]
between the cellularizations of these model categories.
\end{prop}

The second statement in the proposition uses the Cellularization Principle,
\cite[Corollary 2.8]{gscell} with $\cKS_t$ the (derived) image of $\cKS^\torus$ under the Quillen
equivalences of the first part of the statement.

We have shown that the category of rational $\bN$--spectra is equivalent
to the cellularization of modules over a suitable diagram of
commutative DGAs in $\bQ [W]$--modules. On the other hand, we know very little about the
diagram except its homology and that the terms are commutative. The
purpose of this section is to show that this determines the diagram up
to equivalence.

The structure of the argument is precisely the same  as in
\cite[Section 9]{tnqcore}, but  we need to ensure  that the maps may be taken
to be $W$--equivariant.

All the homologies are constructed from cohomology rings $H^*(BK)$ for
$K\subseteq T$. As a ring this is polynomial,  but it has an
action of the subgroup $W_K$ of $W$ fixing $K$.  Taking both
structures into account, it is the symmetric algebra on a finite
dimensional rational  representation $U_K$ of $W_K$:
$$H^*(BK)=\mathrm{Symm}(U_K). $$
Whenever $H^*(BK)$ occurs in $\RRa:=H_*(\RRt)=H_*(\Theta \RRtop)$ it comes in
a product over subgroups of that dimension, and in that product
all conjugates $H^*(BK^w)$ also occur. Taking into account the $W$
action this gives the $\Q W$--algebra
$$\Hom_{W_K}(W, H^*(BK))=\Hom_{W_K}(W, \mathrm{Symm (U_K)}). $$

\begin{lem}\label{lem:eqform}
The $W$-twisted commutative ring $\Hom_{W_K}(W, \mathrm{Symm (U_K)})$ is
 strongly intrinsically formal in the sense that if $A$ is another
 such DGA with this cohomology, then there is a homology
 isomorphism
$$\Hom_{W_K}(W, \mathrm{Symm (U_K)}) \stackrel{\simeq}\lra A. $$

\end{lem}

\begin{proof}
First note that $H_*(A)=\Hom_{W_K}(W, \mathrm{Symm (U_K)})$ has a quotient
module $\Hom_{W_K}(W, U_K)$, and hence there is an epimorphism
$$Z(A)\lra Z(A)/B(A)=H_*(A)\lra \Hom_{W_K}(W, U_K). $$
By Maschke's Theorem this is split and we may choose a submodule
$$\Hom_{W_K}(W, U_K)\subseteq Z(A).  $$
The $\bQ W_K$--module map $U_K\lra Z(A)$ extends to a  map
$$\mathrm{Symm}(U_K) \lra Z(A)$$
of commutative $W_K$--algebras which we may extend to $\bQ [W]$--algebra map
$$\Hom_{W_K}(W, \mathrm{Symm (U_K)}) \lra Z(A)$$
inducing a homology isomorphism as required.
\end{proof}

After this lemma, the rest of the strategy of \cite{tnqcore} can be implemented in the
same way. There are two points worth commenting on. Firstly,
that whenever a factor $H^*(BK)$ occurs, so do all
$H^*(BK^w)$. Secondly, whenever a set $\mcE_K$  of Euler classes
is inverted, we may use representations fixed by $W_K$ (for example by inducing
the representations $V$ of $T$ with $V^K=0$ to $\Wt_K$).
We illustrate the method with the two smallest examples.

\begin{example}[The torus of rank 1]

Note that in this case the only possible Weyl group $W$ is of order
2. This necessarily normalizes every subgroup, so that every term in the product
$\prod_nH^*(B\torus/C_n)$ is invariant. The only change brought about
by the action is that $W$ negates the polynomial generator in the
terms $H^*(B\torus/C_n)$.

After that, the argument proceeds precisely as in the torus case from \cite{tnqcore}. We start with the cofibrant diagram $\RRt$ of commutative rings as in the top row. Extending along the top left hand
vertical we form the second row. The upward maps from the two outer vertices
of $\RRa$ on the bottom row
can then be defined. The Euler classes are defined by the image of $\RRa(0,1)$, and those are
inverted to form the third row, after which the middle vertical can be
filled in.
$$\diagram
\RRt \dto  &\RRt  (0,1) \rto \dto & \RRt (1,1) \dto & \RRt(1,0) \lto
\dto\\
\RRhat_1 \dto &\prod_i\RRt (0,1)_i \rto \dto_= & \prod_i \RRt(0,1)_i
\tensor_{\RRt (0,1)}\RRt (1,1) \dto &\RRt (1,0) \lto \dto^=\\
\RRhat_2& \prod_i\RRt (0,1)_i \rto  & \cEi \prod_i \RRt(0,1)_i
\tensor_{\RRt (0,1)}\RRt (1,1)  &\RRt (1,0) \lto \\
\RRa \uto &\cOcF \rto \uto & \cEi \cOcF \uto &\bQ \lto \uto\\
&&&\\
\RRtop& (D\efp)^\torus\rto &(\siftyV{\torus}\sm D\efp)^\torus&(\siftyV{\torus})^\torus\lto\\
\enddiagram$$
\end{example}

\begin{example}[The torus of rank 2]

In this case there are various possible Weyl groups, and most of them
permute both the finite subgroups and those of dimension 1. For
example if $G=SU(3)$, the Weyl group $W\cong \Sigma_3$ acts on $L\torus$ as
the reduced natural representation. Amongst subgroups of the maximal
torus, the only proper non-trivial subgroup fixed by $W$ is the central
subgroup of order 3, and in most cases, where $W_K\neq1$ it acts non-trivially on
$H^*(B\torus/K)$. Nonetheless, this additional structure is entirely
compatible with the formality argument from \cite{tnqcore}.

It is too typographically complicated to display the full argument in
the way we did for rank 1, but it still seems worth displaying $\RRa$
and $\RRttop$ (recall that $\RRtop=\RRttop^\torus$).
In the process of proving formality we will need to change individual rings in the diagram. Such a change at a given place $i$ of a diagram affects rings in places of the diagram that receive a map from $i$.
For example a change of the ring at the top vertex only affects the three other points not on the bottom face, and then the change of the ring at the middle vertex on the bottom face only affects the central vertex.

\resizebox{0.95\textwidth}{!}{
$$
\diagram
&&\prod_F\bQ[c,d] \drto \dlto&&\\
&\prod_H \cEi_H\prod_F\bQ[c,d] \drto && \cEi_G\prod_F\bQ[c,d] \dlto &\\
&&\cEi_G \prod_H \cEi_H\prod_F\bQ[c,d] &&\\
\prod_H\prod_{\tilde{H}}\bQ[c] \uurto \rrto &&
\cEi_G \prod_H\prod_{\tilde{H}}\bQ [c] \uto &&\bQ \llto \uulto
\enddiagram
$$
}

\vspace{2ex}
The subgroups  $F$ run through finite subgroups, the subgroups $H$ run
through circle subgroups, and the subgroups $\tilde{H}$ run through
subgroups with identity component $H$. The polynomial rings $\bQ [c,d]$
are the cohomology rings of $B(G/F)$ (all different but isomorphic),
and the polynomial rings $\bQ [c]$ are the cohomology rings of
$B(G/\tilde{H})$. The polynomial ring $\bQ$ is the cohomology ring of
$B(G/G)$.

The above diagram is obtained by taking homotopy groups of the
following diagram $\RRttop$  of ring $G$--spectra which we display below.

\vspace{2ex}

\resizebox{0.95\textwidth}{!}{

$$
\diagram
&&D\efp \drto \dlto&&\\
&\prod_H \siftyV{H}\sm D\efp \drto &&      \siftyV{G}\sm D\efp \dlto &\\
&&\siftyV{G}\sm \prod_H \siftyV{H}\sm D\efp &&\\
\prod_H\siftyV{H}\sm D\efhp \uurto \rrto &&
\siftyV{G}\sm \prod_H\siftyV{H}\sm D\efhp \uto &&\siftyV{G} \llto \uulto
\enddiagram
$$
}
\vspace{3ex}
\end{example}

In the case of toral $\bN$--spectra we follow the above strategy using
Lemma \ref{lem:eqform} where necessary.

\begin{prop}\label{prop:alg-formality}
The formality argument above gives a zig-zag of
maps of diagrams of commutative rings with $W$--action,
that are objectwise homology equivalences, between
\[
\RRa =\pi^\torus_*(\RRttop)=\pi_*(\RRtop)=H_*(\RRt)
\]
and $\RRt$.
Hence there is a zig-zag of Quillen equivalences
\[
\RRt \modin \bQ[W] \leftdmod \quillen
\RRa \modin \bQ[W] \leftdmod.
\]
\end{prop}

Once again we can apply the derived functors of the above zig-zag
of Quillen equivalences to obtain a set of cells $\cKS_a$ and we get the following
corollary.

\begin{cor}
There is a zig-zag of Quillen equivalences
\[
\cKS_t \cell \RRt \modin \bQ[W] \leftdmod \quillen
\cKS_a \cell \RRa \modin \bQ[W] \leftdmod.
\]
between the cellularizations of the model categories
in Proposition \ref{prop:alg-formality}.
\end{cor}

\section{The algebraic models}

The remainder of the paper focuses on simplifying the category
$\cKS_a \cell \RRa \modin \bQ[W] \leftdmod$. We begin by introducing
the algebraic structures needed for this simplification.
More details can be found in
\cite{tnqcore}, \cite{tnq3} and \cite{AGtoral}.

Over the years different, but equivalent algebraic models were defined
for the category of rational torus spectra. The point of this section
is not to provide details on them, as that was already done in
\cite{tnqcore} and \cite{tnq3}, but to give some intuition on how they
are built.  In summary this section states that all the constructions
of \cite{tnqcore} are well behaved and compatible with the additional Weyl group
actions. Although this is true and rather elementary, it takes some
time to and notation to explain. Accordingly, we have cut this section to the minimum,
leaving only enough for the reader to be able to cross-check with
\cite{tnq3} and  verify claims made elsewhere in this paper.

The important thing to keep in mind is that an algebraic model for
rational torus spectra is a \textbf{special} full subcategory of the
category of modules over a diagram of rings. Variations on the
algebraic model come from different choices of the shape of this
diagram (and corresponding changes in the rings) and we will discuss some of the possibilities below.
All the equivalences of categories linking different choices of the algebraic models are given by exact functors and thus having an injective model structure on one of the choices of the algebraic model gives Quillen equivalent injective model structures on all the others, using the left-lifting technique of \cite{HKRS} and the fact that cofibrations are exactly the monomorphisms and weak equivalences are exactly homology isomorphisms.

When we are interested in the algebraic model for rational toral
$\bN$--spectra (where the maximal torus $\torus$ is normal in $\bN$
with $W=\bN/\torus$) then all different choices of the algebraic
models for the torus can be ``upgraded'' to include the $W$--action
and model rational toral $\bN$--spectra (see \cite{AGtoral}). Some of
these models are easier to construct than the others and that has to
do with the fact that $W$ in some cases might act on the diagram of
rings itself. However our approach is to start from the easiest case,
where $W$ fixes the objects of the diagram and acts on the ring at each point of the diagram. This category is denoted by $d \acal_d^f(\bN, toral)$ (here the subscript and superscript indicate how the diagram of rings is built, as we explain below).

As in the case of a torus, the injective model structure again exists on all equivalent categories. It is in fact left-lifted from the case of the torus using the forgetful functor forgetting the $W$--action and landing in the algebraic model for a torus. These forgetful functors all preserve limits and colimits and since all the categories discussed here are locally presentable they have both adjoints, thus they can be used to left-lift the model structures.

Unfortunately, the situation gets much more complicated when one is interested in the algebraic model for rational toral $G$--spectra, where the maximal torus $\torus$ is not necessary normal in $G$. The paper \cite{AGtoral} presents such a model, but this time only one option for the diagram of rings is available, due to the complications coming from different Weyl group actions acting in different places of a diagram. This algebraic model is denoted by $d \acal_a^f(G, toral)$.

The strategy is to give an algebraic model for rational toral $\bN$--spectra and then use the final passage of restriction and coinduction both in topology and in algebra to provide an algebraic model for rational toral $G$--spectra. To do that step in algebra however, both categories have to be built using the same shape of diagram of rings, because only then the comparison functors are defined in \cite{AGtoral}. Thus we need to use a different (but equivalent and Quillen equivalent) algebraic model for rational toral $\bN$--spectra, to the one mentioned above, namely  $d \acal_a^f(\bN, toral)$.  For this reason we recall below several different algebraic models for rational $\torus$--spectra and rational toral $\bN$--spectra and adjoint pairs between them. We begin with the case of a torus.

\subsection{Models for the torus}
Following \cite{tnq3}, we introduce a number of terms, leading to the algebraic model for the torus.
The basic idea is that we are considering categories of modules over diagrams of rings.
We consider a number of diagram shapes.
\begin{itemize}
\item $\Sigma_c$ the diagram of connected subgroups of $\bT$.
\item $\Sigma_d$ the poset $\{ 0,1, \dots, r \}$.
\item $\Sigma_a$ the category of all subgroups of $\bT$ with cotoral inclusions, i.e. the inclusions $K\lra H$ such that $H/K$ is a torus.
\end{itemize}

The subscript ``c'' goes for connected, ``d'' for dimension and ``a'' for all subgroups.

We make contravariant diagrams of rings on these categories.
\begin{itemize}
\item $R_a$ given by $R_a(G/K) = H^*(B \bT/K)$.
\item $R_c$ given by $R_c(G/K) = \prod_{\bar{K} \in \mathcal{F}/K}  H^*(B \bT/\bar{K})$.
\item $R_d$ given by $R_d(m) = \prod_{\dim(\bar{K})= m} H^*(B \bT/\bar{K})$.
\end{itemize}
The maps in these diagrams are all built from the maps
$H^*(B\torus/K) \to H^*(B\torus/L)$ for $L \subseteq K \subseteq \torus$.

For each of the above posets, we have an associated poset of flags:
finite length sequences of elements in decreasing order.
We denote these by $\Sigma_c^f$,  $\Sigma_d^f$ and $\Sigma_a^f$.
We also have poset diagrams of pairs (flags of length 2)
$\Sigma_c^p$,  $\Sigma_d^p$ and $\Sigma_a^p$.
We may also extend our diagrams of rings to flag posets (and pair posets), we refer the reader to \cite{tnq3} for details.

In the earlier sections our diagram was always
a punctured cube of dimension $d$, we can now recognise this as
$\Sigma^f_d$, the poset of flags on non-empty subsets
of $\{0,1, \dots, r \}$ and the diagram of rings $\RRa$ as
$R^f_d$.

To obtain the algebraic model, we add restrictions on to the kinds of modules that we consider.
The starting point is that modules must be \textbf{q}uasi-\textbf{c}oherent and
\textbf{e}xtended (qce) and
$\mathcal{F}$-continuous. These are substantial restrictions on the
type of objects that can arise. The precise definitions of these terms
are given in \cite{tnq3}, but understanding these terms is not required
for the current paper.

\begin{rem}Once we fix a diagram of rings, say  $\Sigma^f_a$, the algebraic model for the rational $\torus$--spectra is given by the full subcategory on $\mathcal{F}$-continuous qce differential modules over the diagram of rings $R^f_a$. We denote it by $d \acal_a^f(\bT)$. Thus we have
\[
d \acal_a^f(\bT) = qce \endash R_a^f \leftdmod
\]
\end{rem}

Similarly, we can define other algebraic models for the case of the torus.
We summarise the various categories that will appear in later sections. For more details (including what ``pqce'' means below) see \cite{tnq3}.
\begin{itemize}
\item $d \acal_a^p(\bT)$ is the category of $\mathcal{F}$-continuous qce
differential modules over $R_a^p$, where the diagram is pairs in $\Sigma_a$.
\item $d \acal_a^f(\bT)$ is the category of $\mathcal{F}$-continuous qce
differential modules over $R_a^f$, where the diagram is flags in $\Sigma_a$.
\item $d \acal_c^p(\bT)$ is the category of qce
differential modules over $R_c^p$, where the diagram is pairs in $\Sigma_c$.
\item $d \acal_c^f(\bT)$ is the category of qce
differential modules over $R_c^f$, where the diagram is flags in $\Sigma_c$.
\item $d \acal_d^f(\bT)$ is the category of pqce
differential modules over $R_d^f$, where the diagram is flags in $\Sigma_d$.
\end{itemize}

As discussed at the start of the section, we need to know that the
algebraic models for the torus are locally presentable.

\begin{lem}\label{lem:locPresT} All algebraic models for the torus are locally presentable categories.
\end{lem}
\begin{proof} 
We prove that $ \acal_c^p(\bT)$ is locally presentable. 
This category is abelian by \cite{tnq3} so it suffices to demonstrate that the
category has a {\em set} of generators
$T_i$. That is to say, given two objects $X,Y$ and two maps $f, g: X\lra Y$ with $f\neq g$
there is an $i$ and a map  $T_i\lra X$ so that the composite with the two maps $f $ and $g$ are distinct.

By quasi-coherence, maps are determined by the part at the trivial group $1$. In fact we will show that
there is a set of modules $T_i$ so that any element $x\in X(1)$ lies in a submodule isomorphic to some $T_i$.
In fact we will show that $x$ lies in a submodule $T$ of $X$ which is {\em $\phi$-pointwise countably generated} in the sense that for each connected subgroup $K$ of $\torus$  that $\phi^KT$ is a finitely generated
module over $\cOcFK$. Because of quasi-coherence, this involves a little care, and we return to it after
making an estimate. We will not attempt to be economical in our estimates!

\vspace{2ex}
\noindent\textbf{Claim 1:}  The collection of $\phi$-pointwise countably generated objects of $ \acal_c^p(\bT)$ up to isomorphism forms a set.
\vspace{2ex}

If $R$ is any ring (such as $\cOcFK$) the collection of modules $M$ generated by $n$-elements is a set, because there is a short exact sequence
$$\bigoplus_{i\in I} R\stackrel{p}\lra \bigoplus_1^nR \lra M\lra 0. $$
Since $R$ is a set, the number of submodules is a set, and so we may use subsets of this as our indexing sets $I$. The collection of possible maps is $\prod_iR$, and hence  a set.
Taking the union of these sets as $n$ runs through the natural numbers and a countable set we find a set of countably generated $R$-modules.

At each point $K$ in the diagram this argument applies to show there is a set of countably generated $\cOcFK$-modules $\phi^K T$. There are countably many subgroups $K$ so a choice for a $\phi^K T$ at each $K$ is still a set. Finally, these are related. In fact if  $L\subseteq K$ there is a $\cOcFL$-map
$$\phi^L T\lra \cEi_{K/L}\cOcFL\tensor_{\cOcFK}\phi^KT. $$
There is only a set of such maps. Of course only a subset will satisfy the quasi-coherence and extendedness condition, but we just needed an estimate.

\vspace{2ex}
\noindent\textbf{Claim 2:}  Given any object $X$, element $x \in X(1)$ lies in a $\phi$-pointwise countably generated subobject of $X$.
\vspace{2ex}

We start with $x\in X(1)$ and we will construct a $\phi$-pointwise finitely generated submodule
$Z\subseteq X$ containing $x\in X(1)$. For each connected subgroup $K$ we have
$$\beta_1^K(x) =\Sigma_{i=1}^{n_K}\lambda^K_i\tensor x^K_i$$
for suitable numbers $n_K$ and  elements $\lambda^K_i \in \cEi_K\cOcF, x^K_i \in \phi^KX$. Roughly speaking we want to use the $x^K_i$ as our first guess for the generators of $\phi^KZ$, but to allow for the possibility of some cancellation, we do something a little more elaborate.
In fact we will get possibly different sets of generators for each complete flag $F=(1=K_0\subset K_1 \subset \cdots \subset K_s=K)$ ending at $K$. Note that by completeness,  $K_i$ is of dimension $i$.
Now we proceed one step at a time with
$$\beta_1^{K_1}(x) =\Sigma_{i=1}^{n_{F,1}}\lambda^{F,1}_i\tensor x^{F,1}_i$$
and for each term $x^{F,1}_i$
$$\beta_{K_1}^{K_2}(x^{F,1}_i) =\Sigma_{j=1}^{n_{F,2}}\lambda^{F,2,i}_j\tensor x^{F,2}_j, $$
where $n_{F,2}$ is taken large enough to deal with all the finite number of terms $x^{F,1}_i$.  The remaining steps are similar. The first  approximation to $\phi^KZ$ is the countably generated $\cOcFK$-submodule
$$\phi^KZ'=\langle x^{F,i}_j\st F, i,j\rangle \subseteq \phi^KX,  $$
where $F$ runs through all flags from $1$ to $K$, $1\leq i \leq dim K$ and
$1\leq j \leq n_{F,i}$. To obtain something extended we take $Z'(K)=\cEi_K\cOcF \tensor_{\cOcFK}\phi^KZ'$. There are maps between these modules, but the result is not generally quasi-coherent, so we will add some more generators.

By quasi-coherence of $X$, note that for any $x \in \phi^KX$ and any $L\subseteq K$, there is a representation
$V=V(y,L,K)$ of  $\torus/L$ so that $V^K=0$ and $e(V)x$ is in the image of
$$\beta_L^K: \phi^LX \lra \cEi_{K/L}\cOcFL \tensor_{\cOcFK}\phi^KX. $$
We may therefore choose $y_L^K(V)\in \phi^L X$ with $\beta_L^Ky_L^K=e(V) x$. We use this to increase
the size of the modules $Z'(L)$.

In our case, for each $F,i,j$ we choose $V(F,i,j)$ and $y(F,i,j)$ with
\[
\beta_{K_i}^{K_{i+1}}(y(F,i,j))=e(V(F,i,j))x^{F, i+1}_{j}.
\]
Note that adding the generators $y(F, i,j)$ to $\phi^{K_i}Z'$ does not increase the size of $\cEi_{K/L}Z'(K_i)$ since $x^{F,{i+1}}_{j}$ is already present. However it does ensure that  $\beta_L^K$ becomes surjective after the inversion of Euler classes.

We then take the countably generated $\cOcFL$-module
$$Z(L)=Z'(L)+\langle y(F,i,j)\st 1\leq i\leq n_L, K\geq L, 1\leq j\leq n_{F,i}\rangle. $$
The subobject $Z$ is now quasi-coherent, $\phi$-countably generated and contains $x\in Z(1)$.

This completes the proof of Claim 2, and hence shows $ \acal_c^p(\bT)$ is locally presentable.
Applying the functor $D^1 \otimes (-)$ to the generators of $ \acal_c^p(\bT)$
gives a set of generators for $d \acal_c^p(\bT)$.
As all other models are equivalent via exact equivalences, the result holds. 
\end{proof}

\begin{rem}
In the case of rank 1, \cite{s1q} constructed so-called `wide
spheres', which have the property that given $x\in X(1)$ there is a
wide sphere $WS$ and a map $WS\lra X$ so that $x$ is contained in the
image. Wide spheres are visibly $\phi$-{\em finitely} generated, and give a
fairly explicit set of generators. We do not know if the set of
generators can be taken to be $\phi$-finitely generated in rank $\geq 2$.

The combinatorics of the footprint of an element $x\in X(1)$ is sufficiently complicated, that even enumerating the possible shapes of wide spheres in general is rather daunting.
\end{rem}



\subsection{Models for the normalizer}

The material in this section is a summary of \cite[Section 6]{AGtoral}.
From now on, we will use $\square$ to indicate any of the possible subscripts and superscripts
described in the previous section.

When working in the case of $\bN$, the Weyl group $W=\bN/ \torus$ acts on the diagrams $\Sigma_\square$ of the previous section,
with the trivial action on $\Sigma_d$.
The diagram of rings comes with an action of $W$ that is compatible with this action on the poset. That is, given a diagram of rings $\mcR$ over a poset $\Sigma_{\square}$ and $w \in W$, there is a diagram of rings $w^*\mcR $ defined by $w^*\mcR (E)  = \mcR(E^w)$ and maps of diagrams
$w_{\mcR} : \mcR \to w^* \mcR$ satisfying the expected rules on composition and units.

Note that as each object of $\Sigma_d$ is fixed by $W$, the value of a
diagram of rings $\mcR$ at some element will be a
ring-object in differential graded $\bQ$--modules
with an action of $W$. Whereas on $\Sigma_a$ at each point of a diagram indexed by a subgroup $K$ one will
only have an action of the subgroup of $W$ that stabilises $K$.

We may then define the category of modules over these diagrams of rings,
denoted $R_\square^\square \leftmod$.  As with rings we may ask that the modules are
$W$--equivariant so there are maps $w_M : M \to w^* M$, where
$w^* M (E) = M(E^w)$ but with the action of $R_\square^\square(E)$ given by
$r \cdot m = (r^w)m$. As usual, the maps $w_M$ should compose
appropriately with   the identity $1_M$.
We denote this category
$R_\square^\square \leftmod ^*[W]$
and describe it as $W$--equivariant modules over $R_\square^\square$.
We may then consider such modules with differentials (they are naturally graded
as the rings are all graded), denoted by
$R_\square^\square \leftdmod$ and
$R_\square^\square \leftdmod ^*[W]$.
A more thorough description of the categories with $W$ action is given in
\cite[Section 4.A]{AGtoral}.

We now wish to describe the category $\RRa \modin \bQ[W] \leftdmod$
in terms of these equivariant diagrams.
This will make it easier to compare our work with the case of a torus.

\begin{lem}
The category of modules over $R^f_d=\RRa$ in $\bQ[W] \leftdmod$
is equivalent to the category of
$W$--equivariant diagrams in $R^f_d \leftdmod$
\[
\RRa \modin \bQ[W] \leftdmod
=
R^f_d \modin \bQ[W] \leftdmod
\cong
R^f_d \leftdmod^*[W].
\]
\end{lem}
\begin{proof}
If $R$ is a ring in a category $\bC$ with a $W$ action, the category
$R\leftmod^*[W]$ has objects $R$--modules $M$ with the additional structure of maps
\[
w_m: M\lra w^*M
\]
where the $R$--module $w^*M$ has the $R$--action
$r(w^*m)= w^*[(r^{w^{-1}})m]$.
This is exactly the same as saying that $M$ is an $R$--module in the category
of $W$--objects in $\bC$. It is important here that $W$ acts objectwise on objects in the right hand side category.
\end{proof}

\begin{rem}\label{rem:differentWmodelsForN} All additional conditions
  which go into defining algebraic models as full subcategories on
  special $R$--modules (like pqce) are
compatible with the  action of $W$ described above and thus one can define $W$--equivariant versions of the algebraic models for the rational torus spectra. This is done in detail in \cite{AGtoral} and we list the categories we will use below.

\begin{itemize}
\item $d \acal_a^p(\bN, toral)=d \acal_a^p(\bT)^*[W]$ is the category of $W$--equivariant $\mathcal{F}$-continuous qce
differential modules over $R_a^p$, where the diagram is pairs in $\Sigma_a$.
\item $d \acal_c^f(\bN, toral)=d \acal_c^f(\bT)^*[W],$ is the category of $W$--equivariant  qce
differential modules over $R_c^f$, where the diagram is flags in $\Sigma_c$.
\item $d \acal_d^f(\bN, toral)=d \acal_d^f(\bT)^*[W]$ is the category of $W$--equivariant pqce
differential modules over $R_d^f$, where the diagram is flags in $\Sigma_d$.
\end{itemize}
\end{rem}

\begin{lem}
The forgetful functor
\[
U \colon
R_\square^\square \leftdmod ^*[W]
\longrightarrow
R_\square^\square \leftdmod
\]
is faithful and has both adjoints.
\end{lem}
\begin{proof}
An object on the left hand side is an object on the right hand side with additional structure
given by the action maps for all $w\in W$. The functor $U$ simply forgets this additional structure.
The forgetful functor has both adjoints since it is a functor between locally presentable categories which preserves all limits and colimits.
\end{proof}

\begin{lem}\label{lem:ForgetWactionNmodels} The forgetful functor
\[
U \colon
d\mcA_\square^\square (\bN, toral)
=
d\mcA_\square^\square (\bT)^*[W]
\longrightarrow
d\mcA_\square^\square (\torus)
\]
is faithful and has both adjoints, where the category $d \acal_\square^\square(\bN, toral)$ stands for any of the categories described in Remark \ref{rem:differentWmodelsForN} and $d\mcA_\square^\square (\torus)$ is its non-equivariant, torus version.
\end{lem}
\begin{proof} This follows from the previous result and the fact that both left and right adjoints preserve additional conditions on objects (like pqce or  $\mathcal{F}$-continuous) thus they restrict and corestrict to the algebraic models on both sides.
\end{proof}

\begin{lem} All categories $d \acal_\square^\square(\bN, toral)$ described in Remark \ref{rem:differentWmodelsForN} are locally presentable categories.
\end{lem}
\begin{proof} These are  abelian categories where filtered colimits are exact. One uses Lemma \ref{lem:ForgetWactionNmodels} and Lemma \ref{lem:locPresT} to show that there exists a categorical generator in $d \acal_\square^\square(\bN, toral)$.
\end{proof}

\subsection{Relations between the models}\label{subsec:relations}

Since we mentioned different algebraic models for rational $\torus$--spectra it is time to describe the relationships between them. We present here only the sketch of equivalences between these categories, more details can be found in \cite[Section 10.A and 10.B]{tnq3}.
In effect, Subsections \ref{subsec:relations}, \ref{subsec:forgetful} and \ref{subsec:torsion} amount to observing that the constructions there are
compatible with the $W$--action.
\[
\xymatrix@C+0.3cm{
d \acal_a^f(\bT)
\cong
d \acal_a^p(\bT)
\ar@<-1ex>[r]_-{\Gamma q_!^d}^-{\cong}
&
\ar@<-1ex>[l]_-{e'}
d \acal_c^p(\bT)
\ar@<-1ex>[r]_-{p}^-{\cong}
&
\ar@<-1ex>[l]_-{f}
d \acal_c^f(\bT)
\ar@<-1ex>[r]_-{d_!^e}^-{\cong}
&
\ar@<-1ex>[l]_-{e}
d \acal_d^f(\bT)
}
\]

We briefly introduce the functors $d_!^e$ and $e$ from the diagram above 
and another functor $i$.
Given an object $M \in d \acal_d^f(\bT)$, we can define
a object $eM$ of $d \acal_c^f(\bT)$ as follows.
Let $E=(K_0 \supset \dots \supset K_s) \in \Sigma_c^f$ with dimension vector
$\underline{d} = (d_0 > \dots > d_s) \in \Sigma_d^f$. Then
\[
eM(K_0 \supset \dots \supset K_s) =
e_E M(d_0 > \dots > d_s)
\]
where $e_E$ is the idempotent given by the projection
$R_d^f (\underline{d}) \to R_c^f (E)$
as the first term is an iterated product which has a particular
factor defined by $E$.
The functor
\[
i : d \acal_c^f(\bT) \to R_c^f \leftdmod
\]
is simply the inclusion and the functor
\[
d_!^e : d \acal_c^f(\bT) \to d \acal_d^f(\bT)
\]
is a subfunctor of
\[
\begin{array}{c}
d_* : R_c^f \leftdmod \to R_d^f \leftdmod \\
d_* (N) : \underline{d} \mapsto
\underset{F \in \Sigma_c^f, \dim(F) = \underline{d}}{\prod}  N(F).
\end{array}
\]

\begin{rem} One of the most important features of all the functors above is that they are exact. Combining this with the fact that they are equivalences of categories ensures that they all preserve and create both monomorphisms and homology isomorphisms. This will be crucial for establishing model structures later in the paper.
\end{rem}

\begin{rem}
It is clear from the definition that $d_!^e$, $i$ and $e$ commute with the $W$--action and thus extend to the level of models for rational toral $\bN$--spectra.

\[
\xymatrix@C+0.3cm{
d \acal_a^f(\bN, toral)
\ar@<-1ex>[r]_{p \Gamma q_!^d}
&
\ar@<-1ex>[l]_{e' f}
d \acal_c^f(\bN, toral)
\ar@<-1ex>[r]_{d_!^e}
&
\ar@<-1ex>[l]_e
d \acal_d^f(\bN, toral)
}
\]

\end{rem}

Note that we avoid using diagrams
based on pairs $\Sigma_\square^p$ for $\bN$, as these do not
work correctly in the toral case (normalizers are not functorial in
subgroups, whereas they are functorial in flags).

\subsection{Forgetful functors}\label{subsec:forgetful}

We consider the forgetful functors
from toral models of $\bN$--spectra to models for the torus.
\[
\xymatrix@C+0.3cm{
d \acal_a^f(\bN, toral)
\ar@<-1ex>[r]_{p \Gamma q_!^d}
\ar[d]|U
&
\ar@<-1ex>[l]_{ e'f}
d \acal_c^f(\bN, toral)
\ar@<-1ex>[r]_{d_!^e}
\ar[d]|U
&
\ar@<-1ex>[l]_e
d \acal_d^f(\bN, toral)
\ar[d]|U
\\
d \acal_a^f(\bT)
\ar@<-1ex>[r]_{p \Gamma q_!^d}
&
\ar@<-1ex>[l]_{e'f}
d \acal_c^f(\bT)
\ar@<-1ex>[r]_{d_!^e}
&
\ar@<-1ex>[l]_e
d \acal_d^f(\bT)
}
\]
We define $U$ to be the functor which forgets the action of the Weyl group:
\[
\adjunction{U}{R_d^f \leftdmod^*[W]}{R_d^f \leftdmod}{\oplus_W}.
\]
This functor has a right adjoint (analogous to induction) which sends a module $M$ to the equivariant module which at a flag $F$ takes a value of a skewed product of $W$ with $M$,
$\oplus_W M(F) = \oplus_{w \in W} (w^* M)(F)$, where  $w^* M(F)$ is $M(F^w)$ but with
the action given by $\mu \circ (w \otimes 1) : \R_d^f(F) \otimes M(F^w) \to M(F^w)$.
The functor $\oplus_W$ is also left adjoint to $U$ as finite
coproducts and products coincide (just as induction and coinduction
agree for finite groups).
These functors preserve various structures like p, qce and $\mcF$-continuous
and hence pass from categories of modules to the algebraic models.

In particular, an object $M$ of category $R_\square^f \leftdmod^*[W]$
can be described as an object $UM$ of $R_\square^f \leftdmod$ with extra structure:
maps $UM \to w^* UM$ in the category $R_\square^f \leftdmod$
for each $w \in W$ (such that these maps should be unital and compatible with composition).

\begin{rem}\label{rmk:retract}
The fact that $\oplus_W$ is both the left and right adjoint to the forgetful functor $U$
allows us to obtain a retraction using the counit and unit:
\[
M
\to
\oplus_W  U M
\to
M.
\]
This is a more complicated version of the following retraction for modules over $\bQ[W]$:
\[
\xymatrix@R=1pc{
A \ar[r] &
\hom(\bQ[W], UA) \ar[r] &
\oplus_W UA \ar[r] &
A \\
a \ar@{|->}[r]  &
(w \mapsto w a)
&
(w,m)_{w \in W}
\ar@{|->}[r]  &
\sum_{w \in W} \frac{1}{|W|} wm\\
&
f \ar@{|->}[r]  &
(w, f(w^{-1}))_{w \in W}
}
\]
where $\oplus_W M= \bQ[W] \otimes M$ is the direct sum of $|W|$--copies of $M$, with $W$--action
permuting the factors.
\end{rem}

\subsection{The torsion functor}\label{subsec:torsion}

The aim of this section is to complete the following diagram
by constructing the adjunction marked with dashed arrows so that $U$ commutes with vertical left and right adjoints.
\[
\xymatrix@C+1cm{
R_d^f \leftdmod^*[W]
\ar@<+0ex>[r]|{U}
\ar@<+1ex>@{-->}[d]^{\Gamma^f_d}
&
R_d^f \leftdmod
\ar@<+1ex>[d]^{\Gamma^f_d}
\\
\ar@<+1ex>@{-->}[u]^{d_* i e}
d \acal_d^f(\bN, toral)
\ar@<+0ex>[r]|{U}
&
\ar@<+1ex>[u]^{d_* i e}
d \acal_d^f(\bT)
}
\]

To do this, we first refer to the construction of this adjunction at the torus level (no $W$ action) and then lift it from there.
We present a diagram that contains the various categories and functors we will need. The back square describes the situation without $W$ action.
\[
\xymatrix@C+0.1cm@R+0.1cm{
d \acal^f_c(\bT)
\ar@<-1ex>[rr]_{d^e_!}
\ar@<-1ex>[d]_(0.55){i}
& &
d \acal^f_d(\bT)
\ar@<-1ex>[ll]_{e}
\\
R^f_c \leftdmod
\ar@<+1ex>[rr]^{d_*}
\ar@<-1ex>[u]_(0.45){\Gamma_c^f}
& &
R^f_d \leftdmod
\ar@<+1ex>[ll]^{e}
\\
&
d \acal^f_c(\bT)^*[W]
\ar@<-1ex>[rr]_{d^e_!}
\ar@<-1ex>@{-->}[d]_{i}
\ar[uul] |(0.7)U
& &
d \acal^f_d(\bT)^*[W]
\ar@<-1ex>[ll]_{e}
\ar[uul] |(0.7)U
\\
&
R^f_c \leftdmod^*[W]
\ar@<+1ex>[rr]^{d_*}
\ar@<-1ex>@{-->}[u]_{\Gamma_c^f}
\ar[uul] |(0.7)U
& &
R^f_d \leftdmod^*[W]
\ar@<+1ex>[ll]^{e}
\ar[uul] |(0.7)U
}
\]

The starting place is the top--left adjunction.
\begin{lem}\cite[Theorem 11.1]{AGtoral} 
There is an adjoint pair \[\adjunction{i}{d \acal_c^f(\bT)}{R_c^f\leftdmod}{\Gamma_c^f}\] 
where $i$ is the inclusion functor and $\Gamma_c^f$ is its right adjoint. 
\end{lem}

The adjunction $(d_* i e,\Gamma_d^f)$ was defined in \cite[Section 11.C]{tnq3} going around the solid arrows of the back square. To be more precise,  $\Gamma_d^f$ at the $\bT$
level was defined by passing round the diagram $\Gamma_d^f = d_!^e \Gamma_c^f e$.
With this definition, $\Gamma_d^f$ has a left adjoint $d_* i e$.

We plan to do the same in the front square describing the situation with the $W$ action. To do that, we need to show that the dashed functors exist. It is clear that
the dashed functor
\[
i : d \acal^f_c(\bT) \to R_c^f \leftdmod
\]
(which forgets the pqce structure)
commutes precisely with the functors $w^*$.
It follows that $\Gamma_c^f$ commutes with $w^*$ up to a natural isomorphism. In fact, checking the definition from \cite[Section 11]{tnq3} one can show that it commutes precisely.

Now we may extend $\Gamma_c^f$ to $R_c^f \leftdmod^*[W]$ as follows.
Let $M$ be an object of $R_c^f \leftdmod^*[W]$, we define an element of
$d \acal^f_c(\bT)^*[W]$ as follows.
At $E$ we take $(\Gamma_c^f UM)(E)$ with (skewed) $W$ action given by
\[
\Gamma_d^f U (w_M) (E): (\Gamma_d^f UM)(E)
\longrightarrow
(\Gamma_d^f U w^*M)(E)
=
w^*(\Gamma_d^f UM)(E).
\]
It is routine to check that this defines a skewed $W$--action and hence
we get an object of $d \acal^f_c(\bT)^*[W]$.
We then define $\Gamma_d^f$ at the level of diagrams with $W$--action
by passing round the diagram $\Gamma_d^f = d_!^e \Gamma_c^f e$.
This has adjoint $d_* i e$.

\section{Model structures and Quillen equivalences}
Collecting results of the previous section we have a commuting diagram, where the two bottom vertical adjunctions on both sides are
equivalences of categories. The horizontal functors $U$ commute with both left and right vertical functors. Recall that $\RRa$ from Section \ref{sec:spectratoDGAsformal} is the
diagram of rings $R_d^f$.
\begin{figure}[!ht]
\caption{Diagram of model categories}\label{fig:algebradiagram}
\[
\xymatrix@C+1cm{
R_d^f \leftdmod^*[W]
\ar@<+1ex>[r]|{U}
\ar@<+1ex>[d]^{\Gamma^f_d}
&
\ar@<+1ex>[l]
\ar@<-3ex>[l]
R_d^f \leftdmod_{ii}
\ar@<+1ex>[d]^{\Gamma^f_d}
\\
\ar@<+1ex>[u]^{d_* i e}
d \acal_d^f(\bN, toral)
\ar@<+1ex>[r]|{U}
\ar@<-1ex>[d]_e
&
\ar@<+1ex>[u]^{d_* i e}
\ar@<+1ex>[l]
\ar@<-3ex>[l]
d \acal_d^f(\bT)_{i}
\ar@<-1ex>[d]_e
\\
\ar@<-1ex>[u]_{d_!^e}
d \acal_c^f(\bN, toral)
\ar@<+1ex>[r]|{U}
\ar@<-1ex>[d]_{e'f}
&
\ar@<-1ex>[u]_{d_!^e}
\ar@<+1ex>[l]
\ar@<-3ex>[l]
d \acal_c^f(\bT)
\ar@<-1ex>[d]_{e'f}
\\
\ar@<-1ex>[u]_{p \Gamma q_!^d}
d \acal_a^f(\bN, toral)
\ar@<+1ex>[r]|{U}
&
\ar@<-1ex>[u]_{p \Gamma q_!^d}
\ar@<+1ex>[l]
\ar@<-3ex>[l]
d \acal_a^f(\bT)
\\
}
\]
\end{figure}
The aim of this section is to establish model categories on the left hand side of this diagram.
We will keep adding information about model structures and Quillen equivalences to this diagram as we proceed. At the starting point there are model structure only on
$R_d^f \leftdmod_{ii}$ and $d \acal_c^f(\bT)_{i}$ which is indicated by subscripts $ii$ and $i$. The model structures on both categories are injective model structures, where cofibrations are objectwise monomorphisms and weak equivalences are objectwise homology isomorphisms.

We refer to $R_d^f \leftdmod_{ii}$ as having the doubly--injective model structure
as this model structure is the diagram--injective model structure, 
see \cite[Section 3]{gsmodules}, on diagrams of 
categories of chain complexes each equipped with the injective model structure.

A key input to our approach is that the functors in the diagram above are exact,
with the exception of the two functors labelled $\Gamma^f_d$.

\subsection{First row of the diagram}

\begin{lem}\label{lem:right-proper-doubly-inj}
Any category of generalised diagrams of $R$--modules indexed on a punctured cube with the 
doubly--injective model structure is right proper.
\end{lem}

\begin{proof}
The projective model structure on $R$--modules is right proper and there is a doubly--projective model structure on generalised diagrams of $R$--modules. The doubly--projective 
model structure has fibrations and weak equivalences defined objectwise.
As pullbacks are constructed objectwise,  it is right proper. 
Every fibration in the doubly--injective model structure is in particular a fibration in the doubly--projective model structure.
\end{proof}
It follows from the lemma above that the category of $R_d^f \leftdmod_{ii}$ is right proper.

\begin{prop} Both categories $R_d^f \leftdmod^*[W]$ and $R_d^f \leftdmod$ are locally presentable.
\end{prop}
\begin{proof} We show that $R_d^f \leftdmod^*[W]$ is locally presentable, the other follows similarly.
It is clearly an abelian category. Recall that the forgetful functor $R_d^f \leftdmod^*[W] \lra R_d^f \leftdmod$ is faithful and $R_x$ is a categorical generator for $R_x \leftmod$. Using the notation $K_x$ for the left adjoint to the evaluation functor at $x$, $\textrm{ev}_x: R_\bullet \leftmod \lra R_x\leftmod$ and $\oplus_{W}(-)$ the left adjoint to $U$ the categorical generator is given by $\oplus_{x}K_x\oplus_{W}R_x$.

\end{proof}

\begin{lem}There exists a left-induced model structure on the category $R_d^f \leftdmod^*[W]$ from the doubly--injective model structure $R_d^f \leftdmod_{ii}$ using the adjunction
\[
\xymatrix@C+1cm{
R_d^f \leftdmod^*[W]
\ar@<+1ex>[r]^{U}
&
\ar@<+1ex>[l]
R_d^f \leftdmod_{ii}.}
\]
We use the notation $R_d^f \leftdmod^*[W]_{ii}$ for this new model structure.
The adjunction above becomes a Quillen pair and the left adjoint preserves all fibrations, and thus fibrant objects.
\end{lem}
\begin{proof}This follows from \cite[Theorem 2.2.1]{HKRS}. Since both categories are locally presentable and  $R_d^f \leftdmod_{ii}$ is a cofibrantly generated model structure with all objects cofibant we just need to check that there is a good cylinder object for any object in $R_d^f \leftdmod^*[W]$.

The cylinder object in $R_d^f \leftdmod_{ii}$ is given by tensoring with $\bQ\oplus \bQ\lra Cyl(\bQ) \lra \bQ$ (Notice that the tensoring is done objectwise). The same construction gives a cylinder object in $R_d^f \leftdmod^*[W]_{ii}$.

Notice that the lifted model structure on $R_d^f \leftdmod^*[W]$ is the doubly--injective model structure, because $U$ forgets the $W$--action so it creates monomorphisms and homology isomorphisms.

It is clear that the left adjoint $U$ is a left Quillen functor. What we need to show is that it also preserves all fibrations. To do that we will show that its left adjoint is a left Quillen functor. Take a cofibration $f$ in $R_d^f \leftdmod_{ii}$. Applying left adjoint $L$ and then $U$ sends $f$ to $\oplus_{|W|}f$ which is a cofibration. Thus $L(f)$ is a cofibration by the definition of lifted model structure on $R_d^f \leftdmod^*[W]$. Same argument works for an acyclic cofibration. Thus $L$ is a left Quillen functor and $U$ preserves all fibrations.
\end{proof}

\begin{prop}
The doubly--injective model structure on $R_d^f \leftdmod^*[W]_{ii}$ is right proper.
\end{prop}
\begin{proof}
The right lifted model structure on $R_d^f \leftdmod^*[W]$ from $R_d^f \leftdmod_{ii}$ also exists, and is right proper as
$R_d^f \leftdmod_{ii}$ is right proper.  In this case every fibration in the left-lifted model structure is in particular a fibration in the right lifted model structure, which finishes the proof.
\end{proof}

Because both model structures  $R_d^f \leftdmod_{ii}$ and $R_d^f \leftdmod^*[W]_{ii}$ are right proper we can cellularise both categories. We choose to cellularise the first one at the derived images of the cells $\bN/K_+$, where $K\leq T$ in $G$ and the right hand side at the set of cells $\torus/K_+$, for every $K\leq \torus$. We call these sets
$\bN/\Tsub_+=\cKS_a$ and $\torus/\Tsub_+$ respectively.
Recall that $\RRa$ from Section \ref{sec:spectratoDGAsformal} is the
diagram of rings $R_d^f$.

\begin{lem}\label{lem:preservefibs}
The adjunction
\[
\xymatrix@C+1cm{
\bN/\Tsub_+ \cell R_d^f \leftdmod^*[W]_{ii}
\ar@<+1ex>[r]^{U}
&
\ar@<+1ex>[l]^{\oplus_W}
\torus/\Tsub_+ \cell R_d^f \leftdmod_{ii}.}
\]
is a Quillen pair, where both functors
preserve all cofibrations, fibrations
and weak equivalences.
Moreover $U$ reflects all weak equivalences.
\end{lem}
\begin{proof}
To show that $U$ is a left Quillen functor we will use the Cellularization Principle \cite[Theorem 2.7]{gscell} and an observation that the model structures
 \[
\torus/\Tsub_+ \cell R_d^f \leftdmod_{ii}
\quad \quad
U (\bN/\Tsub_+) \cell R_d^f \leftdmod_{ii}
\]
are the same. This follows from the fact that for $K \leq \torus$,
$U(\bN/K_+)=\oplus_{|W|} \torus/K_+$, where $\bN/K_+$ and $\torus/K_+$ denote the derived images of these cells in the algebraic model and $W=\bN/T$ is a finite group.
In fact, $\bN/K_+=\oplus_{W} \torus/K_+$, hence
\[
[\bN/K_+ ,A]^W \cong [\torus/K_+ ,UA]
\]
where the left hand side denotes maps in the homotopy category of
$R_d^f \leftdmod^*[W]_{ii}$ and the right hand side is
maps in the homotopy category of $R_d^f \leftdmod_{ii}$.
It follows immediately that $U$ preserves and reflects all weak equivalences.

The left adjoint of $U$ and the right adjoint of $U$ are the same:
a direct sum of finitely many copies of the input with a suitable $W$--action.
This functor preserves all fibrations, cofibrations
and weak equivalences as $U$ does.
\end{proof}

\subsection{Right hand-side of the diagram}
By \cite[Proposition 11.3]{tnqcore} there is an injective model structure on the category
$d \acal_c^f(\bT)_{i}$ with weak equivalences the homology isomorphisms and cofibrations the monomorphisms. (The category used in the reference is $\acal_c^p(\bT)$,
which is equivalent to $\acal_d^f(\bT)$ by \cite[Corollary 10.1]{tnq3},
it is routine to check that the equivalences are exact.)

The two adjoint pairs linking $d \acal_c^f(\bT)$ with $d \acal_a^f(\bT)$ and $d \acal_d^f(\bT)$ are equivalences of categories where all the functors are exact.
Using these functors we can transfer the injective model structure $d \acal_c^f(\bT)_{i}$ to both $d \acal_a^f(\bT)$ and $d \acal_d^f(\bT)$.
Since all functors are exact they preserve monomorphisms and homology isomorphisms, moreover they are equivalences of categories so they create both of these classes.
It follows that all three of these categories have injective model structures: ones where the weak equivalences are the homology isomorphisms and cofibrations are the monomorphisms.

Furthermore, the adjunctions linking $d \acal_c^f(\bT)$ with $d \acal_a^f(\bT)$ and $d \acal_d^f(\bT)$ are Quillen equivalences.

The fact that $e$ on the right hand side is also a right Quillen functor (since $d_!^e$ is exact) and  \cite[Proposition 11.5]{tnqcore} implies that the remaining
adjunction $(d_* i e, \Gamma^f_d)$ on the right hand side of the diagram is also a Quillen equivalence:

\[
\xymatrix@C+1cm@R+0.5cm{
\bN/\Tsub_+ \cell R_d^f \leftdmod^*[W]_{ii}
\ar@<+1ex>[r]|{U}
\ar@<+1ex>@{-->}[d]^{\Gamma^f_d}
&
\ar@<+1ex>[l]
\torus/\Tsub_+ \cell R_d^f \leftdmod_{ii}
\ar@<+1ex>[d]^{\Gamma^f_d}^{ \ \ \ \ \ \ \ \  QE}
\\
\ar@<+1ex>@{-->}[u]^{d_* i e}
d \acal_d^f(\bN, toral)
\ar@<+1ex>[r]|{U}
\ar@<-1ex>[d]_e
&
\ar@<+1ex>[u]^{d_* i e}
\ar@<+1ex>[l]
d \acal_d^f(\bT)_{i}
\ar@<-1ex>[d]_e
\\
\ar@<-1ex>[u]_{d_!^e}
d \acal_c^f(\bN, toral)
\ar@<+1ex>[r]|{U}
\ar@<-1ex>[d]_{e' f}
&
\ar@<-1ex>[u]_{d_!^e}_{ \ \ \ \ \ \ \ \  QE}
\ar@<+1ex>[l]
d \acal_c^f(\bT)_{i}
\ar@<-1ex>[d]_{e' f}
\\
\ar@<-1ex>[u]_{p \Gamma q_!^d}
d \acal_a^f(\bN, toral)
\ar@<+1ex>[r]|{U}
&
\ar@<-1ex>[u]_{ \ \ \ \ \ \ \ \  QE}_{p \Gamma q_!^d}
\ar@<+1ex>[l]
d \acal_a^f(\bT)_{i}
\\
}
\]

\subsection{Left hand-side of the diagram}

By the same argument as in \cite[Section 11]{tnqcore} there exists an injective model structure on the category
$d \acal_c^f(\bN,toral)$ with weak equivalences the homology isomorphisms,
cofibrations the monomorphisms and fibrations the surjections with injective kernel.

The same argument as in the previous subsection shows that there exist injective model structures on the categories $d \acal_d^f(\bN, toral)$ and $d \acal_a^f(\bN, toral)$. With these model structures the adjunctions between them are Quillen equivalences.

With these model structures on the left hand side of the diagram, the horizontal adjunctions are Quillen pairs.
To check this, it is enough to note that the forgetful functors $U$ preserve all (acyclic) cofibrations. In fact, the forgetful functors $U$ in the second, third and fourth rows create all weak equivalences and cofibrations in the model structures on the left hand side of the diagram.

It remains to show the following.
\begin{prop} The adjunction
\[
\xymatrix@C+1cm{
d \acal_d^f(\bN, toral)_i
\ar@<+1ex>[r]^(.35){d_*ie}
&
\ar@<+1ex>[l]^(.65){\Gamma_d^f}
\bN/\Tsub_+ \cell R_d^f \leftdmod^*[W]_{ii}}
\]
is a Quillen equivalence.
\end{prop}
\begin{proof} First we need to show the adjunction is a Quillen pair.
Notice that the forgetful functors commute with both adjoints $d_*ie$ and $\Gamma_d^f$ in the following square.
\[
\xymatrix@C+1cm@R+0.5cm{
\bN/\Tsub_+ \cell R_d^f \leftdmod^*[W]_{ii}
\ar@<+1ex>[r]|{U}
\ar@<+1ex>[d]^{\Gamma^f_d}
&
\torus/\Tsub_+ \cell R_d^f \leftdmod_{ii}
\ar@<+1ex>[d]^{\Gamma^f_d}
\\
\ar@<+1ex>[u]^{d_* i e}
d \acal_d^f(\bN, toral)_{i}
\ar@<+1ex>[r]|{U}
&
\ar@<+1ex>[u]^{d_* i e}
d \acal_d^f(\bT)_{i}}
\]

Take a cofibration (a monomorphism) $f \in d \acal_d^f(\bN, toral)_i$.
As $(d_*ie)\circ U$ is a left Quillen functor, we know that
$U(d_*ie(f))=d_*ie(U(f))$ is a cofibration.
Hence $\oplus_W U(d_*ie(f))$ is a cofibration in $d \acal_d^f(\bN, toral)_i$.
Since $d_*ie(f)$ is a retract of $\oplus_W U(d_*ie(f))$ by Remark \ref{rmk:retract}, it follows that
$d_*ie(f)$ is a cofibration.
For the acyclic cofibrations we again use the relation
$U \circ d_*ie=d_*ie \circ U$ along with the fact that $U$
preserves and reflects weak equivalences at each level (at the top level it follows from Lemma \ref{lem:preservefibs}).
Thus $d_*ie$ is a left Quillen functor.

To show that the adjunction is a Quillen equivalence we need to show that for any $X$ in
$d \acal_d^f(\bN, toral)_i$ and fibrant  $Y$ in $\bN/\Tsub_+\cell R_d^f \leftdmod^*[W]_{ii}$
a map $f: d_*ie(X) \lra Y$ is a weak equivalence in $\bN/\Tsub_+\cell R_d^f \leftdmod^*[W]_{ii}$
if and only if its adjoint $f^\flat: X \lra \Gamma^f_d(Y)$ is a weak equivalence in $d \acal_d^f(\bN, toral)_i$.

A map $f: d_*ie(X) \lra Y$ is a weak equivalence in $\bN/\Tsub_+\cell R_d^f \leftdmod^*[W]_{ii}$
if and only if
\[
U(f): U(d_*ie(X))=d_*ie(U(X)) \lra U(Y)
\]
is a weak equivalence in $d \acal_d^f(\bT)_{i}$.
Since $U$ preserves fibrant objects
by Lemma \ref{lem:preservefibs},
$U(Y)$ is fibrant.
So $U(f)$ is a weak equivalence in $d \acal_d^f(\bT)_{i}$ if and only if its adjoint
\[
U(f)^\flat: U(X) \lra \Gamma^f_d(U(Y))
\]
is a weak equivalence in $\torus/\Tsub_+ \cell R_d^f \leftdmod_{ii}$
because the adjunction $(d_*ie, \Gamma^f_d)$ is a Quillen equivalence on the right hand side of the diagram.
Since $U(f)^\flat=U(f^\flat)$ this happens if and only if $f^\flat$ is a weak equivalence in $d \acal_d^f(\bN, toral)_i$ which finishes the proof.
\end{proof}

Collecting all results from above we obtain the algebraic model for rational toral $\bN$--spectra. In fact, we obtain three equivalent algebraic models in this situation, depending on the choice of the indexing diagram category, which we so far indicated by a subscript and superscript in the notation. We omit this indication in the following summary.

\begin{thm} There is a zig-zag of Quillen equivalences between rational toral $\bN$--spectra and an algebraic model $d  \acal (\bN, toral)$ with the injective model structure:
\[
L_{e_{\torus_\bN}S_{\bQ}}(\NspO) \quillen d \acal (\bN, toral).
\]
 \end{thm}

Now we are ready to pass to the model for rational toral
$G$--spectra. Notice that to obtain an algebraic model for rational
toral $\bN$--spectra it was not necessary to discuss other models than the one indexed on dimension and flags $d \acal_d^f (\bN, toral)$, since that is the one we can relate directly to the passage from topology. However, the indexing diagram of this sort for rational toral $G$--spectra has not been established.
The only model for that discussed in \cite{AGtoral} is indexed on flags and all (toral) subgroups. Thus to link the model for $\bN$--spectra to the model for $G$--spectra using  \cite{AGtoral} we had to consider $d\acal_a^f(\bN, toral)$.

\section{Passage to the algebraic model for \texorpdfstring{$G$}{G}--spectra}
In this section we complete the proof of our main theorem. 

\begin{thm}
\label{thm:toralmodel}
 There is a zig-zag of Quillen equivalences between rational toral $G$--spectra  and an algebraic model $d  \acal_a^f (G, toral)$ with the injective model structure: 
\[
L_{e_{\torus_G}S_{\bQ}}(\GspO) \quillen d \acal_a^f (G, toral), 
\]
where  $ \acal_a^f (G, toral)$ was defined in \cite{AGtoral}.
 \end{thm}

The principal remaining ingredient is the formality of the objects we
use to cellularise, which is proved using the Adams spectral sequence
of \cite{AGtoral}, but there are also some model categorical
foundations before we begin.

\subsection{Proper preliminaries}
Before we proceed, we need to establish the formal framework. 

\begin{lem} The model categories $d \acal_c^f(\bN, toral)_{i}$, $d \acal_a^f(\bN, toral)_i$ and $d \acal_d^f(\bN, toral)_i$ are right proper.
\end{lem}
\begin{proof} In the model structure $d \acal_c^f(\bT)_{i}$ fibrations are in particular objectwise surjections, by construction of the model structure in \cite[Proposition 11.3]{tnqcore} and pullbacks in $d \acal_c^f(\bT)$ are calculated objectwise. Thus $d \acal_c^f(\bT)_{i}$ is right proper and since  $d \acal_c^f(\bN, toral)_{i}$ is the left lifted model structure from $d \acal_c^f(\bT)_{i}$ it is also right proper.
It is clear that right-properness is preserved by equivalences of
categories which are also Quillen equivalences between $d \acal_c^f(\bN, toral)_{i}$ and both $d \acal_a^f(\bN, toral)_i$ and $d \acal_d^f(\bN, toral)_i$. Thus all three categories are right proper.
\end{proof}

Now that we know that every model category on the left hand side of Figure 
\ref{fig:algebradiagram} is right proper, we can cellularise them. To obtain an algebraic model of rational toral $G$--spectra we need to cellularise the algebraic model $d \acal_a^f(\bN, toral)_i$ at the set of derived images in $d \acal_a^f(\bN, toral)_i$ of cells $G/K_+$ where $K\leq \torus$, see Figure \ref{fig:main} in the introduction.

The last step of the comparison is the simplification of this non-explicit algebraic model given by  $G/\Tsub_+ \cell d \acal_a^f(\bN, toral)_i$ to obtain a model $d \acal_a^f(G, toral)$ from \cite{AGtoral} with the injective model structure.

This happens in two stages. The first one is to recognise that the cells given by the derived images of $G/K_+$ where $K\leq \torus$ are formal in $d \acal_a^f(\bN, toral)_i$. That is, each cell is weakly equivalent to its own homology in
$d \acal_a^f(\bN, toral)_i$. We can thus keep track of  the objects we use
to cellularise, since they are identified by their homology alone. Below we will use notation $G/K_+$ for the derived image in $d\acal(\bN, toral)$ or $d\acal(G, toral)$ of topological $G/K_+$.
The second step is recognising that this cellularisation of the algebraic model for toral $\bN$--spectra is Quillen equivalent to the injective model structure on $d \acal_a^f(G, toral)$.

\subsection{Formality of \texorpdfstring{$G/K_+$}{G/K+}}

We want to show that the cells $G/K_+$ are formal in $d \acal_a^f(\bN, toral)_i$. It is a bit simpler to first show that the coinduced `cells' $F_K(G_+, S^0)$ are formal in $d \acal_a^f(G, toral)_i$. We will come back to the original problem in Lemma \ref{lem:GKformal}.

We consider the context in which we have a homology functor $\piA_*$
on a triangulated category, and a convergent Adams spectral sequence
$$\Ext^{s,t}_{\mcA}(\piA_*(X), \piA_*(Y))\Rightarrow [X,Y]_{t-s}. $$
As in \cite{tnq1} the observation is that if $\Ext^{s,t}_{\mcA}(M,M)$ has a
vanishing line of slope 1 (in the sense that it is zero for $s>
t-s$), then $M$ is intrinsically formal. Indeed, if $\piA_*(X)\cong
\piA_*(X')\cong M$
in the Adams spectral sequence
\[
\Ext^{s,t}(\piA_*(X), \piA_*(X'))\Rightarrow [X,X']_{t-s}
\]
the identity map in $\Ext^{0,0}$ is an infinite cycle and hence is
represented by a map $f:X\lra X'$. Since $f_*$ is the identity, it is
an isomorphism and  $f$ is an equivalence by the Whitehead theorem.

We are considering the $G$--spectrum $X=F_K(G_+, S^0)$
where $K\subseteq \torus$. This is coinduced from $\torus$:
\[
F_K(G_+,S^0)\simeq F_\torus(G_+,
F_K(\torus_+,S^0))\simeq F_\torus(G_+,\Sigma^{L(\torus,K)}\torus/K_+)
\]
where $L(\torus,K)=T_{eK}\torus/K$ is the tangent space to $\torus/K$ at the
identity coset.  In  algebra, which is to say in the abelian models
$\mcA (G,toral)$ and $\mcA (\bN,toral)=\mcA(\torus)^*[W]$, we have the
corresponding statement. Indeed induction and corestriction correspond
to the adjunction
$$\adjunction{\theta_*}{\mcA (G,toral)}{\mcA (\bN,toral)}{\Psi}$$
as in \cite[Corollary 7.11]{AGtoral}. That paper also describes both homology functors $\piAG_*$ and $\piAN_*$ which we use below.

\begin{lem} \cite[Proposition 11.12]{AGtoral}
$$\piAG_* (F_K(G_+,S^0))=\Psi \piAN_*(F_K(\bN_+,S^0)). $$
\end{lem}

For brevity, we simplify notation and write
$$M_G(K)=\piAG_* (F_K(G_+, S^0)), $$
so that the lemma states
$$M_G(K)=\Psi M_\bN(K). $$
Recall that  $\piAN_*(X)$ is the same as $\piAT_*(X)$, but with the action of
$W$ remembered.

\begin{prop}
The object $M_G(K)$ has a vanishing line of slope 1, and hence is
intrinsically formal.
\end{prop}

\begin{proof} First we note that by \cite[Lemma 14.2]{AGtoral}, we have
$$\piAN_*(F_\torus(\bN_+, A))\cong \piAT_*(A)[W]. $$
Since $M_\torus(K)$ has a vanishing line of slope 1 by \cite[Corollary
3.39]{tnq1}, it follows that $M_\bN(K)$ has a vanishing line of slope 1.

Considering the proof, the  injective resolution of $M_\torus(K)$ described in \cite{tnq1} has
the property that the $s$th term is a sum of injectives
$\Sigma^{2s}f_K(H_*(B\torus /K))$. It follows that if $N$ is $(-1)$-connected, then
$\Ext^{s,t}(N, M_\torus(K))$ has a vanishing line of slope 1. This applies
with  $N=M_G(K)$.
\end{proof}

Recall from \cite{AGtoral} that $\Psi $ is an exact right adjoint so the
adjunction passes to Ext.
\begin{lem} There are isomorphisms
\[
\Ext_{\mcA(G,toral)}^{s,t}(L,\Psi M) \cong
\Ext_{\mcA(\bN,toral)}^{s,t}(\theta_* L, M) \cong \Ext_{\mcA(\torus)}^{s,t}(\theta_* L, M)^W.
\]
\end{lem}
\begin{proof} Since $\Psi $ is a right adjoint, it
 takes injectives to injectives. Since it is exact,  it takes injective
 resolutions to injective resolutions.
\end{proof}

Applying this lemma to the case in hand, we conclude
$$\Ext_{\mcA(G,toral)}^{s,t}(M_G(K),M_G(K)) \cong
\Ext_{\mcA(\torus)}^{s,t}(\theta_* \Psi M_\bN(K), M_\bN(K))^W. $$

\begin{rem}
To show formality for cells of the form $G/K_+$ in $d \mcA(G,toral)_i$, we start with
$G/K_+=F_\bN(G_+, \Sigma^{L(G,\bN)} \bN/K_+)$ 
(with $L(G, \bN)$ the tangent space representation) and proceed as before. We now
work with
\[
M_\bN(K)=\piAN_*(\Sigma^{L(G,\bN)} \bN/K_+) =\piAT_*(\Sigma^{L(G,\bN)} \bN/K_+)[W]. 
\]
This has vanishing line of slope 1 over $\torus$ because suspension
preserves this property: indeed, we may simply suspend the resolution
and note that the shifts in the suspensions are the same in domain and
codomain and hence preserve the vanishing line.
\end{rem}

Now we come back to formality in $d \mcA (\bN,toral)_i$.

\begin{lem}\label{lem:GKformal} The cells $G/K_+$ are intrinsically formal in $d\mcA (\bN,toral)_i$.
\end{lem}

\begin{proof} We proved that $G/K_+$ is intrinsically formal in $d\mcA
  (G,toral)_i$.  It is of the form $\Psi M_\bN(K)$ for some $M_\bN(K) \in d\mcA (\bN,toral)_i$. Thus it is enough to show that if $\Psi M_\bN(K)$ is intrinsically formal in $d\mcA (G,toral)_i$ then $\theta_* \Psi M_\bN(K)$ is intrinsically formal in $d\mcA (\bN,toral)_i$, since we know that $\theta_*$ corresponds to restriction.

Take $P \in d \mcA (\bN,toral)_i$ with the same homology as $\theta_* \Psi M_\bN(K)$. 
The unit of the adjunction is an isomorphism by \cite[Corollary 7.11]{AGtoral},
so $\Psi \theta_* \Psi M_\bN(K)\cong \Psi M_\bN(K)$. 
Hence $\Psi M_\bN(K)$ and $\Psi P$ have the same homology in $d\mcA (G,toral)_i$ and since $\Psi M_\bN(K)$ was formal, $\Psi M_\bN(K) \simeq \Psi P$. Thus 
$\theta_* \Psi M_\bN(K) \simeq \theta_* \Psi P$. 

Since the unit of the adjunction $(\theta_*, \Psi)$ is an isomorphism, 
the counit is an isomorphism on an object of the form $\theta_* X$, 
by the triangle equality as drawn below.  
\[
\xymatrix@C+1cm@R+0.5cm{
\theta_*\Psi(\theta_* X)
\ar[r]^-{\epsilon_{\theta_* X}}
&
\theta_*X
\ar[dl]^-{=}
\\
\theta_*X
\ar[u]_-{\theta_*\eta}^{\cong}
}
\]
The functor $\theta_*$ is exact, hence
\[
H_* P \cong H_* \theta_* \Psi M_\bN(K) \cong \theta_* H_*\Psi M_\bN(K).
\]
It follows that the counit on $H_* P$ is an isomorphism. 
As $\Psi$ is also exact, it follows that the counit $\theta_* \Psi P \lra P$ 
is a homology isomorphism. Thus we have shown that 
\[
\theta_* \Psi M_\bN(K)
\simeq 
\theta_* \Psi P
\simeq 
P. \qedhere
\]
\end{proof}

\subsection{Algebraic model for rational toral \texorpdfstring{$G$}{G}--spectra}

In this section we omit subscripts $a$ and superscripts $f$ from the notation, since we only consider indexing diagrams built out of flags and all (toral) subgroups.

Before we pass to $d\mcA (G,toral)$ we need two more results, which will allow us to recognise cellular equivalences in $d\mcA (G,toral)$ with respect to the set of cells $G/K_+$ for $K\leq \torus$ as precisely homology isomorphisms.

\begin{lem} All elements of the set of cells $G/K_+$ for $K\leq \torus$ in $d\mcA (G,toral)$ are homotopically compact.
\end{lem}
\begin{proof}First notice that by \cite[Proposition 16.2]{AGtoral} there is a derived functor $$\theta^!:\mcA (\torus) \lra \mcA (G,toral)$$ which models the induction functor on spectra and is left adjoint to $\theta_*$ at the derived level. Since $\theta_*$ is also a left adjoint at the derived level it commutes with coproducts. Thus $G/K_+\cong \theta^!(\torus/K_+)$ in $\mcA (G,toral)$ for any $K \leq \torus$ and since $\torus/K_+$ were homotopically compact in $d\mcA (\torus)$, their images under the derived functor $\theta^!$ are as well. We used here the formality of the cells in both models.
\end{proof}

\begin{lem}\label{lem:generators_for_G} The set $G/\Tsub_+$ of cells $G/K_+$ for $K\leq \torus$ in $d\mcA (G,toral)$ is a set of generators for the injective model structure on $d\mcA (G,toral)$.
\end{lem}
\begin{proof}

We will use here the name ``cellular equivalence'' for a weak equivalence in the model category $G/\Tsub_+\cell d\mcA (G,toral)$ and we will show that cellular equivalences are homology isomorphisms.

By the use of mapping cones, it is enough to show that if an object $X$ is cellularly trivial (i.e. $[G/K_+,X]^{\mcA(G,toral)}=0$ for all $K\leq \torus$) then $H_*X=0$.
Using the derived functor $\theta^!: \mcA (\torus)\lra \mcA (G,toral)$ from \cite[Proposition 16.2]{AGtoral} which models the induction functor on spectra we  get that
\[
0=[G/K_+,X]^{\mcA(G,toral)}=[\theta^!(\torus/K_+),X]^{\mcA(G,toral)}=[\torus/K_+,\theta_*(X)]^{\mcA(\torus)}
\]
and by \cite[Theorem 12.1]{tnqcore} $H_*(\theta_*(X))=0$. Notice that $\theta_*$ is an exact functor and that implies that $H_*(X)=0$, which finishes the proof.
\end{proof}

Finally, we move to $d\mcA (G,toral)$.

\begin{thm}
The adjunction
$$\adjunction{\theta_*}{d\mcA (G,toral)}{d\mcA (\bN,toral)}{\Psi}$$
is a Quillen pair when both categories are considered with the injective model structures.
It becomes a Quillen equivalence after we cellularise the left hand side at $\theta_*(G/K_+)$ where $K \leq \torus$
$$\adjunction{\theta_*}{G/\Tsub_+\cell d\mcA (\bN,toral)}{d\mcA (G,toral)}{\Psi}$$
\end{thm}

\begin{proof}
The adjunction is a Quillen pair with respect to injective model structures on both sides, since both adjoints at the level of abelian categories ($\mcA (G,toral)$ and $\mcA (\bN,toral)$) are exact, thus $\theta_*$ preserves monomorphisms and homology isomorphisms.

Moreover since the unit of the adjunction at the level of abelian categories is an isomorphism by \cite[Corollary 7.11]{AGtoral} and $\Psi$ is exact (and thus preserves all homology isomorphisms) it follows that the derived unit is a weak equivalence.

We use Cellularisation Principle \cite{gscell} to show that this adjunction becomes a Quillen equivalence after cellularisation of the right hand-side at $\theta_*(G/K_+)$ where $K \leq \torus$. By Lemma \ref{lem:generators_for_G}  the set $G/K_+$ where $K\leq \torus$ is a set of generators for the injective model structure on $d\mcA (G,toral)$, so we only need to show that the derived counit on $\theta_*(G/K_+)$ is a weak equivalence. 
This follows from the same argument regarding the counit as the one in the proof of  
Lemma \ref{lem:GKformal}.
%
%
\end{proof}

Finally, collecting all results from above we have completed the proof
of Theorem \ref{thm:toralmodel} to obtain the algebraic model for rational toral $G$--spectra.

\bibliographystyle{alpha}
\bibliography{toralbib}

\begin{thebibliography}{BGKS17}

\bibitem[BGKS17]{BGKSso2}
D.~Barnes, J.~P.~C. Greenlees, M.~K\k{e}dziorek, and B.~Shipley.
\newblock Rational {${\rm SO}(2)$}-equivariant spectra.
\newblock {\em Algebr. Geom. Topol.}, 17(2):983--1020, 2017.

\bibitem[Gre]{Gmfo}
J.~P.~C. Greenlees.
\newblock Triangulated categories of rational equivariant cohomology theories.
\newblock Oberwolfach Reports 8/2006, 480-488.

\bibitem[Gre98]{greratmack}
J.~P.~C. Greenlees.
\newblock Rational {M}ackey functors for compact {L}ie groups. {I}.
\newblock {\em Proc. London Math. Soc. (3)}, 76(3):549--578, 1998.

\bibitem[Gre99]{s1q}
J.~P.~C. Greenlees.
\newblock Rational {$S\sp 1$}-equivariant stable homotopy theory.
\newblock {\em Mem. Amer. Math. Soc.}, 138(661):xii+289, 1999.

\bibitem[Gre08]{tnq1}
J.~P.~C. Greenlees.
\newblock Rational torus-equivariant stable homotopy {I}: {C}alculating groups
  of stable maps.
\newblock {\em J. Pure Appl. Algebra}, 212(1):72--98, 2008.

\bibitem[Gre16a]{AGtoral}
J.~P.~C. Greenlees.
\newblock Rational equivariant cohomology theories with toral support.
\newblock {\em Algebr. Geom. Topol.}, 16(4):1953--2019, 2016.

\bibitem[Gre16b]{tnq3}
J.~P.~C. Greenlees.
\newblock Rational torus-equivariant stable homotopy {III}: {C}omparison of
  models.
\newblock {\em J. Pure Appl. Algebra}, 220(11):3573--3609, 2016.

\bibitem[Gre18]{CouniversalCommutative_greenlees}
J.~P.~C. Greenlees.
\newblock Couniversal spaces which are equivariant commutative ring spectra.
\newblock {arXiv: 1801.09766}, 2018.

\bibitem[GS13]{gscell}
J.~P.~C. Greenlees and B.~Shipley.
\newblock The cellularization principle for {Q}uillen adjunctions.
\newblock {\em Homology Homotopy Appl.}, 15(2):173--184, 2013.

\bibitem[GS14a]{gsfixed}
J.~P.~C. Greenlees and B.~Shipley.
\newblock Fixed point adjunctions for equivariant module spectra.
\newblock {\em Algebr. Geom. Topol.}, 14(3):1779--1799, 2014.

\bibitem[GS14b]{gsmodules}
J.~P.~C. Greenlees and B.~Shipley.
\newblock Homotopy theory of modules over diagrams of rings.
\newblock {\em Proc. Amer. Math. Soc. Ser. B}, 1:89--104, 2014.

\bibitem[GS18]{tnqcore}
J.~P.~C. Greenlees and B.~Shipley.
\newblock An algebraic model for rational torus-equivariant spectra.
\newblock {\em J. Topol.}, 11(3):666--719, 2018.

\bibitem[HKRS17]{HKRS}
K.~Hess, M.~K{\k{e}}dziorek, E.~Riehl, and B.~Shipley.
\newblock A necessary and sufficient condition for induced model structures.
\newblock {\em J. Topol.}, 10(2):324--369, 2017.

\bibitem[Ill83]{illman}
S.~Illman.
\newblock The equivariant triangulation theorem for actions of compact {L}ie
  groups.
\newblock {\em Math. Ann.}, 262(4):487--501, 1983.

\bibitem[K{\k{e}}d17]{KedziorekSO(3)}
M.~K{\k{e}}dziorek.
\newblock An algebraic model for rational {${\rm SO}(3)$}-spectra.
\newblock {\em Algebr. Geom. Topol.}, 17(5):3095--3136, 2017.

\bibitem[MM02]{mm02}
M.~A. Mandell and J.~P. May.
\newblock Equivariant orthogonal spectra and {$S$}-modules.
\newblock {\em Mem. Amer. Math. Soc.}, 159(755):x+108, 2002.

\bibitem[RS17]{richtershipley}
B.~Richter and B.~Shipley.
\newblock An algebraic model for commutative {H{$\mathbb{Z}$}}--algebras.
\newblock {\em Algebr. Geom. Topol.}, 17(4):2013--2038, 2017.

\bibitem[Shi07]{shiHZ}
B.~Shipley.
\newblock {$H \mathbb{Z}$}-algebra spectra are differential graded algebras.
\newblock {\em Amer. J. Math.}, 129(2):351--379, 2007.

\bibitem[tD79]{tomdieck1}
T.~tom Dieck.
\newblock {\em Transformation groups and representation theory}, volume 766 of
  {\em Lecture Notes in Mathematics}.
\newblock Springer, Berlin, 1979.

\end{thebibliography}

\end{document}